\documentclass[oneside,british,english]{amsart}
\usepackage[T1]{fontenc}
\usepackage[latin9]{inputenc}
\usepackage{amsthm}
\usepackage{amstext}
\usepackage{amssymb}
\usepackage{esint}

\makeatletter
\numberwithin{equation}{section}
\numberwithin{figure}{section}
\theoremstyle{plain}
\newtheorem{thm}{\protect\theoremname}
  \theoremstyle{plain}
  \newtheorem{cor}[thm]{\protect\corollaryname}
  \theoremstyle{definition}
  \newtheorem{defn}[thm]{\protect\definitionname}
  \theoremstyle{plain}
  \newtheorem{prop}[thm]{\protect\propositionname}
  \theoremstyle{plain}
  \newtheorem{lem}[thm]{\protect\lemmaname}
  \theoremstyle{remark}
  \newtheorem{rem}[thm]{\protect\remarkname}
  \theoremstyle{definition}
  \newtheorem{example}[thm]{\protect\examplename}

\makeatother

\usepackage{babel}
  \addto\captionsbritish{\renewcommand{\corollaryname}{Corollary}}
  \addto\captionsbritish{\renewcommand{\definitionname}{Definition}}
  \addto\captionsbritish{\renewcommand{\examplename}{Example}}
  \addto\captionsbritish{\renewcommand{\lemmaname}{Lemma}}
  \addto\captionsbritish{\renewcommand{\propositionname}{Proposition}}
  \addto\captionsbritish{\renewcommand{\remarkname}{Remark}}
  \addto\captionsbritish{\renewcommand{\theoremname}{Theorem}}
  \addto\captionsenglish{\renewcommand{\corollaryname}{Corollary}}
  \addto\captionsenglish{\renewcommand{\definitionname}{Definition}}
  \addto\captionsenglish{\renewcommand{\examplename}{Example}}
  \addto\captionsenglish{\renewcommand{\lemmaname}{Lemma}}
  \addto\captionsenglish{\renewcommand{\propositionname}{Proposition}}
  \addto\captionsenglish{\renewcommand{\remarkname}{Remark}}
  \addto\captionsenglish{\renewcommand{\theoremname}{Theorem}}
  \providecommand{\corollaryname}{Corollary}
  \providecommand{\definitionname}{Definition}
  \providecommand{\examplename}{Example}
  \providecommand{\lemmaname}{Lemma}
  \providecommand{\propositionname}{Proposition}
  \providecommand{\remarkname}{Remark}
\providecommand{\theoremname}{Theorem}

\begin{document}

\title{Uniqueness of signature for simple curves}

\author{H. Boedihardjo, H. Ni and Z. Qian}

\address{Oxford-Man Institute of Quantitative Finance and Mathematical Institute,
Oxford.}
\begin{abstract}
We propose a topological approach to the problem of determining a
curve from its iterated integrals. In particular, we prove that a
family of terms in the signature series of a two dimensional closed
curve with finite $p$ variation, $1\leq p<2$, are in fact moments
of its winding number. This relation allows us to prove that the signature
series of a class of simple non-smooth curves uniquely determine the
curves. This implies that outside a Chordal SLE$_{\kappa}$ null set,
where $0<\kappa\leq4$, the signature series of curves uniquely determine
the curves. Our calculations also enable us to express the Fourier
transform of the $n$-point functions of SLE curves in terms of the
expected signature of SLE curves. Although the techniques used in
this article are deterministic, the results provide a platform for
studying SLE curves through the signatures of their sample paths. 
\end{abstract}
\maketitle
\textit{\small{}Keywords}{\small{}:}{\footnotesize{} Rough path theory;
Uniqueness of signature problem; SLE curves. }{\footnotesize \par}

\section{Introduction }

The signature of a path is a formal series of its iterated integrals.
In \cite{Chen54}, K.T. Chen observed that the map that sends a path
to its signature forms a homomorphism from the concatenation algebra
to the tensor algebra and used it to study the cohomology of loop
spaces. Recent interest in the study of signature has been sparked
by its role in the rough path theory. In particular, it was shown
by Hambly and Lyons in \cite{tree like} that for ODEs driven by paths
with bounded total variations, the signature is a fundamental representation
of the effect of the driving signal on the solution.

This article has two purposes:

1. To determine the winding number of a curve from its signature.

2. To prove, using a relation obtained from answering 1., that the
signature of sufficiently regular planar simple curves uniquely determines
the curves. 

The first question was originally considered as far back as 1936,
in a paper by Rado\cite{Greens p equal one}, who observed that the
second term of the signature series of a smooth path is equal to the
integral of its winding number around $\left(x,y\right)$, considered
as a function of $\left(x,y\right)$. In \cite{Yamthesis}, Yam considered
the same problem as ours, but used a different approach. He started
with the formula 
\[
\mbox{Winding number around }z=\frac{1}{2\pi i}\int_{\gamma}\frac{1}{w-z}\mathrm{d}w.
\]
and smoothened the kernel $w\rightarrow\frac{1}{w-z}$ around the
singularity at $w=z$. He then expanded $\frac{1}{w-z}$ into a power
series of $w$ and used the fact that the line integrals along $\gamma$
of polynomials in $w$ can be expressed in terms of the signature
of $\gamma$. 

Here we took a different approach and obtained a formula for the Fourier
transform of the winding number, which appears to be simpler than
the formula for the winding number itself. A classical result about
iterated integrals, first proved by Chen \cite{Chen Lie series },
states that the logarithm of the signature of any path is a Lie series.
The first result of this article states that the coefficients of some
Lyndon basis elements in the log signature series are in fact moments
of the winding number. In what follows, we will use some basic notions
in free Lie algebra, which we shall recall in \foreignlanguage{british}{section
3. Throughout this article, we will use $\pi_{N}$ to denote the projection
of $T\left(\left(\mathbb{R}^{d}\right)\right)$ to $T^{N}\left(\mathbb{R}^{d}\right)$
(see section 2.1) and $S\left(\gamma\right)_{0,1}$ to denote the
signature of $\gamma$. }
\selectlanguage{british}%
\begin{thm}
\label{thm:winding 1}Let $1\leq p<2$. Let $\gamma:\left[0,1\right]\rightarrow\mathbb{R}^{2}$
be a continuous closed curve with finite $p$ variation. Let $\left\{ \mathbf{e}_{1},\mathbf{e}_{2}\right\} $
denote the standard basis of $\mathbb{R}^{2}$. Define an order on
$\left\{ \mathbf{e}_{1},\mathbf{e}_{2}\right\} $ by $\mathbf{e}_{1}<\mathbf{e}_{2}$.
Then 

1. For each $\left(n,k\right)\in\mathbb{N}\times\mathbb{N}$, $\mathbf{e}_{1}^{\otimes n}\otimes\mathbf{e}_{2}^{\otimes k}$
is a Lyndon word in the free Lie algebra generated by $\left\{ \mathbf{e}_{1},\mathbf{e}_{2}\right\} $
with respect to the tensor product. 

2. For each $n,k\in\mathbb{N}\cup\left\{ 0\right\} \times\mathbb{N}\cup\left\{ 0\right\} $,
let $\mathcal{P}_{\mathbf{e}_{1}^{\otimes\left(n+1\right)}\otimes\mathbf{e}_{2}^{\otimes\left(k+1\right)}}$
be the Lyndon element corresponding to the Lyndon word $\mathbf{e}_{1}^{\otimes\left(n+1\right)}\otimes\mathbf{e}_{2}^{\otimes\left(k+1\right)}$.
Then, for all $n,k\in\mathbb{N}\cup\left\{ 0\right\} \times\mathbb{N}\cup\left\{ 0\right\} $,
$N\geq n+k+2$, the coefficient of $\mathcal{P}_{\mathbf{e}_{1}^{\otimes n+1}\otimes\mathbf{e}_{2}^{\otimes k+1}}$
in the Lyndon basis expansion of the truncated log signature $\pi_{N}$$\left(\log S\left(\gamma\right)\right)$
is 

\begin{eqnarray}
 &  & \left(-1\right)^{k}\int_{\mathbb{R}^{2}}\frac{x^{n}y^{k}}{n!k!}\eta\left(\gamma-\gamma_{0},\left(x,y\right)\right)\mathrm{d}x\mathrm{d}y.\label{eq:winding 1}
\end{eqnarray}
where $\eta\left(\gamma-\gamma_{0},\left(x,y\right)\right)$ is the
winding number of the curve $\gamma-\gamma_{0}$ around the points
$x\mathbf{e}_{1}+y\mathbf{e}_{2}$. 
\end{thm}
\selectlanguage{english}%
As the winding number of a path does not contain information about
the order at which it passes through points, whereas signature does,
we cannot expect that the signature of a path can be expressed in
terms of just winding numbers. In particular, let $a$ and $b$ be
two closed curves in $\mathbb{R}^{2}$, both starting at $0$ and
let $\star$ denote the concatenation operation between two paths.
Then $a\star b$ and $b\star a$ have the same winding number around
any point, but in general do not have the same signature. Nevertheless,
it is natural to ask how many terms in the signature series of a path
can be represented in terms of its winding numbers. The answer is
that the first four terms of a closed curve's signature can be expressed
in terms of its winding number. 
\selectlanguage{british}%
\begin{cor}
\label{cor:winding 2}Let $1\leq p<2$. Let $\gamma:\left[0,1\right]\rightarrow\mathbb{R}^{2}$
be a continuous closed curve with finite $p$ variation. The first
four terms of $\log\left(S\left(\gamma\right)_{0,1}\right)$ can be
expressed in terms of the function $\left(x,y\right)\rightarrow\eta\left(\gamma-\gamma_{0},\left(x,y\right)\right)$
alone. 
\end{cor}
\selectlanguage{english}%
At the end of section three, we will prove that the number ``four''
is sharp. In other words, there are two paths $\gamma,\tilde{\gamma}$
which have the same winding number around every point, but the fifth
terms of the signature of $\gamma$ and $\tilde{\gamma}$ differ.
The reason is that all Lyndon words of degree at most $4$ generated
by $\left\{ \mathbf{e}_{1},\mathbf{e}_{2}\right\} $ are of the form
$\mathbf{e}_{1}^{\otimes n}\otimes\mathbf{e}_{2}^{\otimes k}$. On
the other hand, there is a Lyndon word of degree 5 which is not of
the form $\mathbf{e}_{1}^{\otimes n}\otimes\mathbf{e}_{2}^{\otimes k}$,
namely, $\mathbf{e}_{1}\otimes\mathbf{e}_{2}\otimes\mathbf{e}_{1}\otimes\mathbf{e}_{2}^{\otimes2}$.
This corresponds to the difficulty in expressing the iterated integral
\[
\int_{0<s<t<1}\left[\left[\gamma_{s},\mathrm{d}\gamma_{s}\right],\left[\gamma_{t},\left[\gamma_{t},\mathrm{d}\gamma_{t}\right]\right]\right]
\]
in terms of the moments of the winding number of $\gamma$.

\section*{Uniqueness of signature}

If we consider the signature as a representation of paths, then an
interesting question is whether this representation is faithful. This
was first considered by Chen himself \cite{Chen uniqueness}, who
proved that irreducible, piecewise regular continuous paths have the
same signature if and only if they are equal up to a translation and
a reparametrisation. His result was generalised with a new, quantitative
approach by Hambly and Lyons in \cite{tree like} who showed that
two paths $\gamma$ and $\tilde{\gamma}$ with finite total variations
have the same signature if and only if $\gamma$ can be expressed
as the concatenation of $\tilde{\gamma}$ with a ``tree-like'' path
$\sigma$.
\begin{thm}
\label{thm:jordan1}Let $1\leq p<2$. Let $\gamma$, $\tilde{\gamma}$
be simple curves with finite $p$ variation in $\mathbb{R}^{2}$.
Then $S\left(\gamma\right)_{0,1}=S\left(\tilde{\gamma}\right)_{0,1}$
if and only if $\gamma$ and $\tilde{\gamma}$ are equal up to a translation
and a reparametrisation.
\end{thm}
In the case of $p=1$, we already know from the result of Hambly and
Lyons that the simple curves can be recovered from the signature (modulo
translation and reparametrisation) since simple curves have no tree-like
parts. An interesting, but difficult extension is to prove that if
the signatures of two curves with finite $p>1$ variations are equal,
then the paths are equal up to the tree-like path equivalence. The
restriction $1\leq p<2$ gives us the existence of signature for free,
thanks to Young's integration theory. 

Theorem \ref{thm:jordan1} only applies to paths with finite $p$-variations,
where $p<2$. In particular, our results can only be applied to study
stochastic processes whose sample paths are almost surely smoother
than the Brownian motion sample paths. One example of such processes
is the Chordal SLE$_{\kappa}$ measure. The SLE measures were born
from the study of lattice models which have conformally invariant
scaling limit. There are a number of other lattice models whose scaling
limit has been proved to be an SLE curve under some boundary conditions,
such as the loop erased random walk ($\kappa=2$, \cite{LSW Loop erased}),
the Ising model ($\kappa=3$, \cite{Ising}), the level lines of Gaussian
free field ($\kappa=4,$ \cite{Gaussian Free Fields}), percolation
on the triangular lattice ($\kappa=6$, \cite{Percolation scaling}
and \cite{percolation 2}), and the Peano curve of the uniform spanning
tree ($\kappa=8$, \cite{LSW Loop erased}). 

The path regularity and, in particular, the roughness of SLE curves,
in relation to the speed $\kappa$ of the driving Brownian motion,
is an extremely interesting topic. It is intuitively clear that the
SLE curves become rougher as the speed of the driving Brownian motion
increases. In \cite{holder SLE}, the optimal Hölder exponent for
SLE curves under the capacity parametrisation was proved to be
\[
\min\left(\frac{1}{2},1-\frac{\kappa}{24+2\kappa-8\sqrt{8+\kappa}}\right).
\]

In \cite{Hausdorff dimension of SLE}, V. Beffara proved that the
almost-sure Hausdorff dimension of SLE curves is $\min\left(1+\frac{\kappa}{8},2\right)$.
Therefore, the optimal Hölder exponent cannot exceed $\frac{1}{1+\frac{\kappa}{8}}$.
B. Werness\cite{Wer12} proved that for $0<\kappa\leq4$, almost surely,
the SLE curve in $\mathbb{D}$ has finite $p$ variation for any $p>1+\frac{\kappa}{8}$.
In another words, the roughness of an SLE curve grows linearly with
the speed of the driving Brownian motion. It is strongly believed
that this remains true for $4<\kappa<8$. However, to the best of
our knowledge, this problem remains open.

In \cite{Wer12}, B. Werness used his regularity result to define
the signatures of SLE curves using Young's integral. He is also the
first to realise that the Green's theorem can be used to compute some
terms in the signature of a simple curve. He used it to prove the
$n=2,k=1$ case of Lemma \ref{thm:main1} for simple closed curves
and to compute the first three gradings of the expected signature
of SLE curve. Our work is inspired by and in fact generalises Werness's
calculation. Later in Theorem \ref{thm:SLEthem2}, we shall show that
our generalisation allows us to obtain the fourth term in the expected
signature of SLE$_{\kappa}$ curves. Werness's method will not work
to calculate fifth or later terms in the expected signature of SLE
curves. This is because the fifth or later terms are not completely
determined by the path's winding number. 

In the study of SLE curves we often do not care about the curves'
parametrisations and in some cases, it may be convenient to study
the curves' signature instead. In order to do so, one must prove that
there is a $1-1$ correspondence between curves and their signatures,
outside a null set. Such injectiveness was proved for Brownian motion
by Le Jan and Qian in \cite{LQ12} and for general diffusion processes
by Geng and Qian. Both results rely on the Strong Markov property.
Although the Chordal SLE$_{\kappa}$ measure is not Markov, the inversion
problem can be tackled for $\kappa\leq4$ since the Chordal SLE$_{\kappa}$
measure is supported on simple curves. The Chordal $SLE_{\kappa}$
measure in a domain $D$ is defined as the pull-back of the Chordal
SLE$_{\kappa}$ measure in $\mathbb{H}$ via a conformal map. Although
the Chordal SLE$_{\kappa}$ measure in $\mathbb{H}$ is parametrised
on $[0,\infty)$, we know from \cite{rhode and schramm} that the
Chordal SLE$_{\kappa}$ measure in $\mathbb{H}$ is supported on curves
tending to infinity as time tends to infinity. This allows us to reparametise
SLE$_{\kappa}$ curves in a bounded Jordan domain $D$ so that it
is defined on $\left[0,\infty\right]$, or $\left[0,1\right]$, by
continuous extension. It follows from Theorem \ref{thm:jordan1} that:
\begin{thm}
\label{thm:SLEthm} Let $D$ be a bounded Dini-smooth Jordan domain
and let $a,b$ be two distinct boundary points of $D$. Let $\mathbb{P}_{\kappa,D}^{a,b}$
be the Chordal $SLE_{\kappa}$ measure in $D$ with marked points
$a$ and $b$. Then there exists a set of curves $A$, such that $\mathbb{P}_{\kappa,D}^{a,b}\left(A^{c}\right)=0$
for all $0<\kappa\leq4$ and if $\gamma,\tilde{\gamma}\in A$ and
$S\left(\gamma\right)_{0,1}=S\left(\tilde{\gamma}\right)_{0,1}$,
then $\gamma$ and $\tilde{\gamma}$ are equal up to a reparametrisation. 
\end{thm}
The Dini-smooth condition was introduced to ensure the existence of
a Lipschitz conformal map from $\mathbb{D}$ to $D$. See \cite{Pommerenke92}
for a proof of this result and the definition of Dini-smooth. This
ensures that the SLE$_{\kappa}$ curves in $D$ have the same regularity
as the SLE$_{\kappa}$ curves in $\mathbb{D}$.

The expected signature can be considered as the ``Laplace transform''
of a stochastic process and has first been studied in \cite{Fawcett thesis}.
An important open problem in rough path theory is whether one can
recover a probability measure from the expected signature corresponding
to the probability measure. In fact, in the case when the measure
is a Dirac delta measure, this problem is the uniqueness of signature
problem introduced earlier. For the Chordal SLE measure, the sequence
of $n$-point functions, first studied by O. Schramm, describes the
distribution of the winding angle of the SLE curve around any $n$
given points in the interior of the domain. For $\kappa\leq4$, as
the SLE$_{\kappa}$ curve is simple, so knowing the winding angle
around every point is equivalent to knowing the image of the curve.
Therefore, the $n$-point function for all $n$ can be considered
as a parametrisation independent version of the ``finite dimensional
distribution'' of SLE curve. We prove that the Laplace transform
of the $n$ point functions, and hence the $n$ point functions themselves,
can be obtained from the expected signature of SLE curves. 
\begin{thm}
\label{thm:SLEthem2}Let $0\leq\kappa\leq4$. Let $D$ be a bounded
Dini-smooth Jordan domain and $a,b\in\partial D$. Let $\mathbb{P}_{\kappa,D}^{a,b}$
be the Chordal SLE$_{\kappa}$ measure in $D$ with marked points
$a$ and $b$. For each curve $\gamma$, let $\Phi\left(\gamma\right)$
denote the concatenation of $\gamma$ with the positively oriented
arc of $\partial D$ from $b$ to $a$. For each $N\in\mathbb{N}$,
let $\Gamma_{N}$ denote the $n$-point function associated with $\mathbb{P}_{\kappa,D}^{a,b}$,
then for all $N\geq1$ and $\lambda_{i},\mu_{i}\in\mathbb{R}$ for
$i=1,\ldots N$, 

\begin{eqnarray*}
 &  & \int_{\mathbb{R}^{2N}}e^{\sum_{i=1}^{N}\lambda_{i}x_{i}+\mu_{i}y_{i}}\Gamma_{N}\left(\left(x_{1},y_{1}\right),\ldots,\left(x_{N},y_{N}\right)\right)\mathrm{d}x_{1}\cdots\mathrm{d}y_{N}\\
 & = & \sum_{n_{1},\ldots,n_{N},k_{1}\ldots k_{N}\geq0}\Pi_{i=1}^{N}\left(\lambda_{i}\right)^{n_{i}}\left(-\mu_{i}\right)^{k_{i}}\sqcup_{i=1}^{N}\mathbf{e}_{1}^{*\otimes\left(n_{i}+1\right)}\otimes\mathbf{e}_{2}^{*\otimes\left(k_{i}+1\right)}\left(\mathbb{E}_{\kappa,D}^{a,b}\left[S\left(\Phi\left(\cdot\right)\right)_{0,1}\right]\right)
\end{eqnarray*}
where $\mathbf{e}_{i}^{*}$ is the dual basis corresponding to the
standard basis of $\mathbb{R}^{2}$ (see section 2.1) and $\sqcup$
denotes the shuffle product (see Proposition \ref{prop:shuffle-1}).
\end{thm}
Note that if $\mathbb{E}_{\kappa,D}^{a,b}\left(S\left(\cdot\right)_{0,1}\right)$
is the expected signature of SLE curve and $\phi$ is the negatively
oriented arc from $a$ to $b$ in $\partial D$, then 
\[
\mathbb{E}_{\kappa,D}^{a,b}\left(S\left(\cdot\right)_{0,1}\right)=\mathbb{E}_{\kappa,D}^{a,b}\left[S\left(\Phi\left(\cdot\right)\right)_{0,1}\right]\otimes S\left(\phi\right)_{0,1}.
\]
 Theorem \ref{thm:SLEthem2} also allows us to obtain the fourth term
of the expected signature of the Chordal SLE$_{\kappa}$ measure for
$\kappa\leq4$ in terms of the $1$ and $2$ point functions. While
a partial differential equation can be written down for the $2$ point
function, the only case when the $2$ point function is known explicitly
is when $\kappa=\frac{8}{3}$. This allows us to obtain an expression
for the fourth term of the expected signature of SLE$_{\frac{8}{3}}$
curve as: 
\begin{thm}
\label{thm:fourth term of sle signature }The fourth term in the expected
signature of SLE$_{\frac{8}{3}}$ curve in $\frac{1}{2}\left(1+\mathbb{D}\right)$
is 
\[
\frac{\mathbf{e}_{1}^{\otimes4}}{4!}-\left(\frac{5}{96}-\frac{\mathcal{K}}{16}\right)\left[\mathbf{e}_{1},\left[\left[\mathbf{e}_{1},\mathbf{e}_{2}\right],\mathbf{e}_{2}\right]\right]-\frac{1}{8}\left(\frac{5}{6}-\mathcal{K}\right)\left[\left[\mathbf{e}_{1},\mathbf{e}_{2}\right],\mathbf{e}_{2}\right]\otimes\mathbf{e}_{1}+\left(\frac{\pi^{2}}{128}+\frac{A}{2}\right)\left[\mathbf{e}_{1},\mathbf{e}_{2}\right]\otimes\left[\mathbf{e}_{1},\mathbf{e}_{2}\right]
\]
where $\mathcal{K}$ is the Catalan constant $\sum_{i=1}^{\infty}\frac{\left(-1\right)^{k}}{\left(2k+1\right)^{2}}=0.916\ldots$
and $A$ is the quadruple integral 
\begin{equation}
\int_{0}^{\pi}\int_{0}^{\infty}\int_{0}^{\pi}\int_{0}^{\infty}\frac{r_{1}r_{2}\left[\left(1+\cos\theta_{1}\right)\left(1+\cos\theta_{2}\right)+\sin\theta_{1}\sin\theta_{2}G\left(\sigma\right)\right]\mathrm{d}r_{1}\mathrm{d}\theta_{1}\mathrm{d}r_{2}\mathrm{d}\theta_{2}}{4\left(r_{1}^{2}+2r_{1}\sin\theta_{1}+1\right)^{2}\left(r_{2}^{2}+2r_{2}\sin\theta_{2}+1\right)^{2}\left(\cos\theta_{1}+1\right)\left(\cos\theta_{2}+1\right)}\label{eq:quad integral}
\end{equation}
where 
\[
\sigma:=\frac{r_{1}^{2}+r_{2}^{2}-2r_{1}r_{2}\cos\left(\theta_{1}-\theta_{2}\right)}{r_{1}^{2}+r_{2}^{2}-2r_{1}r_{2}\cos\left(\theta_{1}+\theta_{2}\right)}
\]
is the exponential of twice the Green's function in the upper half
plane and 
\[
G\left(\sigma\right)=1-\sigma\,_{2}F_{1}\left(1,\frac{4}{3};\frac{5}{3};1-\sigma\right)
\]
where $_{2}F_{1}$ is the hypergeometric function. 
\end{thm}
The plan for the rest of the article is as follows. 

In section 2, we recall the basic results about the signature and
winding number. 

In section 3, we prove Theorem \ref{thm:winding 1} and Corollary
\ref{cor:winding 2}.

In section 4, we prove Theorem \ref{thm:jordan1}.

In section 5, we prove Theorem \ref{thm:SLEthm}. 

In section 6, we prove Theorems \ref{thm:SLEthem2} and \ref{thm:fourth term of sle signature }.

\subsection*{Acknowledgement }

We thank the anonymous referee for the detailed comments and suggestions.
All three authors acknowledge the support of the ERC (Grant Agreement
No.291244 Esig).

\section{Preliminaries}

\subsection{Basic notations}

\selectlanguage{british}%
Let $T\left(\left(\mathbb{R}^{d}\right)\right)$ be the set of sequences
\[
\left(a_{0},a_{1},a_{2},\ldots\right)
\]
where $a_{i}\in\left(\mathbb{R}^{d}\right)^{\otimes i}$. equipped
with the addition and multiplication operations $+$ and $\otimes$.
The binary operations $+$ and $\otimes$ are defined so that for
all $\mathbf{a},\mathbf{b}\in T\left(\left(\mathbb{R}^{d}\right)\right)$,
if $\pi^{\left(i\right)}$ denotes the projection of a sequence onto
its $i$ th term, then 
\begin{equation}
\pi^{\left(n\right)}\left(\mathbf{a}+\mathbf{b}\right):=\pi^{\left(n\right)}\left(\mathbf{a}\right)+\pi^{\left(n\right)}\left(\mathbf{b}\right)\label{eq:tensor addition}
\end{equation}
and 
\begin{equation}
\pi^{\left(n\right)}\left(\mathbf{a}\otimes\mathbf{b}\right):=\sum_{i=0}^{n}\pi^{\left(i\right)}\left(\mathbf{a}\right)\otimes\pi^{\left(n-i\right)}\left(\mathbf{b}\right).\label{eq:tensor multiplication}
\end{equation}

$T\left(\left(\mathbb{R}^{d}\right)\right)$ is called the formal
series of tensors of $\mathbb{R}^{d}$. 

\selectlanguage{english}%
Let $T^{k}\left(\mathbb{R}^{d}\right)$ denote the set of all finite
$k$-sequences 
\[
\left(a_{0},\ldots,a_{k}\right)
\]
where $a_{i}\in\left(\mathbb{R}^{d}\right)^{\otimes i}$. The addition
and multiplication operations, $+$ and $\otimes$, on $T^{k}\left(\mathbb{R}^{d}\right)$
are defined by (\ref{eq:tensor addition}) and (\ref{eq:tensor multiplication})
for $n=0,1,\ldots,k$. We will use $\pi_{k}$ to denote the projection
map from $T\left(\mathbb{R}^{d}\right)$ to $T^{k}\left(\mathbb{R}^{d}\right)$. 

\selectlanguage{british}%
For each $f_{1},\ldots,f_{k}\in\left(\mathbb{R}^{d}\right)^{*}$ define
$f_{1}\otimes\ldots\otimes f_{k}$ on $\left(\mathbb{R}^{d}\right)^{\otimes k}$
by extending linearly the relation 
\[
f_{1}\otimes\ldots\otimes f_{k}\left(v_{1}\otimes\ldots\otimes v_{k}\right):=f_{1}\left(v_{1}\right)\ldots f_{n}\left(v_{k}\right).
\]
We may extend the map $f_{1}\otimes\ldots\otimes f_{k}$ to a functional
on $T\left(\left(\mathbb{R}^{d}\right)\right)$ by defining for all
$\mathbf{a}\in T\left(\left(\mathbb{R}^{d}\right)\right)$, 
\[
f_{1}\otimes\ldots\otimes f_{k}\left(\mathbf{a}\right):=f_{1}\otimes\ldots\otimes f_{k}\left(\pi^{\left(k\right)}\left(\mathbf{a}\right)\right).
\]

\selectlanguage{english}%

\subsection{Signature }

Let $p>1$ and let $\mathcal{V}^{p}\left(\left[0,1\right],\mathbb{R}^{d}\right)$
denote the set of all continuous functions $\gamma:\left[0,1\right]\rightarrow\mathbb{R}^{d}$
such that 
\begin{equation}
\left\Vert \gamma\right\Vert _{\mathcal{V}^{p}\left(\left[0,1\right],\mathbb{R}^{d}\right)}^{p}:=\sup_{\mathcal{P}}\sum_{k}\left|\gamma_{t_{k+1}}-\gamma_{t_{k}}\right|^{p}<\infty.\label{eq:variation}
\end{equation}
where the supremum is taken over all finite partitions $\mathcal{P}:=\left(t_{0},t_{1},..,t_{n-1},t_{n}\right)$,
where $0=t_{0}<t_{1}<...<t_{n-1}<t_{n}=1$. 

The elements of $\mathcal{V}^{p}\left(\left[0,1\right],\mathbb{R}^{d}\right)$
will be called curves with finite $p$ variation. This class of paths
with finite $p$ variation is narrower than the one used by Young
\cite{Young} because we restrict our considerations to \textit{continuous}
paths. 

Note that $\left\Vert \cdot\right\Vert _{\mathcal{V}^{p}\left(\left[0,1\right],\mathbb{R}^{d}\right)}$
defines a semi-norm on $\mathcal{V}^{p}\left(\left[0,1\right],\mathbb{R}^{d}\right)$. 
\begin{defn}
\label{sig  defn}Let $1\leq p<2$. Let $\gamma\in\mathcal{V}^{p}\left(\mathbb{R}^{d}\right)$and
let 
\[
\triangle_{n}\left(s,t\right):=\left\{ \left(t_{1},\ldots,t_{n}\right):s<t_{1}<\cdots<t_{n}<t\right\} .
\]
The \textit{lift} of $\gamma$ is a function $S\left(\gamma\right)_{\cdot,\cdot}:\left\{ \left(s,t\right):0\leq s\leq t\right\} \rightarrow T\left(\left(\mathbb{R}^{d}\right)\right)$
defined by 
\begin{equation}
S\left(\gamma\right)_{s,t}=1+\sum_{n=1}^{\infty}\int_{\triangle_{n}\left(s,t\right)}\mathrm{d}\gamma_{t_{1}}\otimes\ldots\otimes\mathrm{d}\gamma_{t_{n}}\label{eq:signature definition}
\end{equation}
where the integrals are taken in the sense of Young \cite{Young}. 

The signature of a path $\gamma\in\mathcal{V}^{p}\left(\mathbb{R}^{d}\right)$
on $\left[0,1\right]$ is defined to be $S\left(\gamma\right)_{0,1}$. 
\end{defn}
We shall use the following properties of signature, whose proofs can
be found in \cite{St-flours} or \cite{Friz-victoir}.

1. (Invariance under reparametrisation)For any $t\in[0,\infty)$,
$S\left(\gamma\right)_{0,t}$ is invariant under any reparametrisation
of $\gamma$ on $\left[0,t\right]$. 

2. (Inverse) $S\left(\gamma\right)_{0,1}\otimes S\left(\overleftarrow{\gamma}\right)_{0,1}=\mathbf{1}$,
where $\overleftarrow{\gamma}\left(t\right):=\gamma\left(1-t\right)$
is the reversal of $\gamma$ and $\mathbf{1}$ is the identity element
in $T\left(\mathbb{R}^{d}\right)$.

3. (Chen's identity)$S\left(\gamma\right)_{s,u}\otimes S\left(\gamma\right)_{u,t}=S\left(\gamma\right)_{0,t}$
for any $0\leq s<u<t\leq1$

4. (Scaling and translation)Let $\lambda\in\mathbb{R}^{d}$, $\mu\in\mathbb{R}$,
then 
\[
S\left(\lambda+\mu\gamma\right)_{s,t}=1+\sum_{n=1}^{\infty}\mu^{n}\int_{\triangle_{n}\left(s,t\right)}\mathrm{d}\gamma\left(t_{1}\right)\otimes...\otimes\mathrm{d}\gamma\left(t_{n}\right)
\]

5.(Lie series) log$S\left(\gamma\right)_{0,1}$ is a Lie series. 

6.(Shuffle product formula) We define a \textit{(r,s)-shuffle} to
be a permutation of $\left\{ 1,2,...,r+s\right\} $ such that $\sigma\left(1\right)<\sigma\left(2\right)<..<\sigma\left(r\right)$
and $\sigma\left(r+1\right)<...<\sigma\left(r+s\right)$\textit{ . }
\begin{prop}
\label{prop:shuffle-1}(\cite{St-flours},Theorem 2.15) Let $1\leq p<2$
and $\gamma\in\mathcal{V}^{p}\left(\left[0,1\right],\mathbb{R}^{d}\right)$,
then 
\begin{eqnarray*}
 &  & \mathbf{e}_{k_{1}}^{*}\otimes\ldots\otimes\mathbf{e}_{k_{r}}^{*}\left(S\left(\gamma\right)_{0,1}\right)\mathbf{e}_{k_{r+1}}^{*}\otimes\ldots\otimes\mathbf{e}_{k_{r+s}}^{*}\left(S\left(\gamma\right)_{0,1}\right)\\
 & = & \sum_{\left(r,s\right)-shuffles\;\sigma}\mathbf{e}_{k_{\sigma^{-1}\left(1\right)}}^{*}\otimes\ldots\otimes\mathbf{e}_{k_{\sigma^{-1}\left(r+s\right)}}^{*}\left(S\left(\gamma\right)_{0,1}\right).
\end{eqnarray*}
where $\cdot$ is the multiplication operation in $\mathbb{R}$. 
\end{prop}
The sum 
\[
\sum_{(r,s)-shuffles\;\sigma}\mathbf{e}_{k_{\sigma^{-1}\left(1\right)}}^{*}\otimes\ldots\otimes\mathbf{e}_{k_{\sigma^{-1}\left(r+s\right)}}^{*}
\]
is denoted by $\mathbf{e}_{k_{1}}^{*}\otimes\ldots\otimes\mathbf{e}_{k_{r}}^{*}\sqcup\mathbf{e}_{k_{r+1}}^{*}\otimes\ldots\otimes\mathbf{e}_{k_{r+s}}^{*}$.

We shall need a few approximation theorems relating the $p$-variation
of a path with its piecewise linear interpolations. For a continuous
function $\gamma$ and a partition $\mathcal{P}:=t_{0}=0<t_{1}<..<t_{n}=1$,
the piecewise linear interpolation of $\gamma$ with respect to $\mathcal{P}$
is defined as the following function on $\left[0,T\right]$: 
\[
\gamma_{t}^{\mathcal{\mathcal{P}}}:=\gamma_{t_{i}}+\left(\frac{\gamma_{t_{i+1}}-\gamma_{t_{i}}}{t_{i+1}-t_{i}}\right)\left(t-t_{i}\right)\;\mbox{for}\; t\in\left[t_{i},t_{i+1}\right]
\]

Then the following approximation theorem holds:
\begin{lem}
(Lemma 1.12 and Proposition 1.14, \cite{St-flours})\label{lem:cont2}Let
$p$ and $q$ be such that $1\leq p<q$. Let $\gamma\in\mathcal{V}^{p}\left(\left[0,1\right],\mathbb{R}^{d}\right)$.
Then for all finite partitions $\mathcal{P}$, 
\[
\left\Vert \gamma^{\mathcal{P}}\right\Vert _{\mathcal{V}^{p}\left(\left[0,1\right],\mathbb{R}^{d}\right)}\leq\left\Vert \gamma\right\Vert _{\mathcal{V}^{p}\left(\left[0,1\right],\mathbb{R}^{d}\right)}
\]

Furthermore for all $\varepsilon>0$, there exists a $\delta>0$ such
that for all partitions $\mathcal{P}$ of $\left[0,1\right]$ satisfying
$\left\Vert \mathcal{P}\right\Vert <\delta$ we have 
\begin{eqnarray*}
\left\Vert \gamma-\gamma^{\mathcal{\mathcal{P}}}\right\Vert _{\mathcal{V}^{q}\left(\left[0,1\right],\mathbb{R}^{d}\right)} & < & \varepsilon,\:\mbox{and}\\
\sup_{t\in\left[0,1\right]}\left\Vert \gamma_{t}-\gamma_{t}^{\mathcal{\mathcal{P}}}\right\Vert  & < & \varepsilon.
\end{eqnarray*}

\end{lem}
The following lemma is extremely useful in proving the properties
of Young's integral.
\begin{lem}
\label{lem:approx3}Let $\gamma:\left[0,1\right]\rightarrow\mathbb{R}^{d}$
be a continuous curve with finite $p$-variation, where $p<2$. Let
$\mathcal{P}_{m}$ be a sequence of partitions such that $\mathcal{P}_{m}$
contains both $0$ and $1$ for all $m$ and $\left\Vert \mathcal{P}_{m}\right\Vert \rightarrow0$
as $m\rightarrow\infty$. For any $\left(i_{1},\ldots,i_{n}\right)\in\left\{ 1,\ldots,d\right\} ^{n}$,
\begin{equation}
\mathbf{e}_{i_{1}}^{*}\otimes\ldots\otimes\mathbf{e}_{i_{n}}^{*}\left[S\left(\gamma\right)_{0,1}\right]=\lim_{m\rightarrow\infty}\mathbf{e}_{i_{1}}^{*}\otimes\ldots\otimes\mathbf{e}_{i_{n}}^{*}\left[S\left(\gamma_{s}^{\mathcal{P}_{m}}\right)_{0,1}\right].\label{eq:sigapprox}
\end{equation}
\end{lem}
\begin{proof}
See Corollary 2.11 in \cite{St-flours}. 
\end{proof}
\selectlanguage{british}%

\subsection{Winding number }

\selectlanguage{english}%
In this section, we shall recall the definition of winding number
and a few key basic facts that we shall use. 
\begin{defn}
Let $\gamma:\left[0,1\right]\rightarrow\mathbb{R}^{2}$ be a continuous
function. Then

1.$\gamma$ is a closed curve if $\gamma_{0}=\gamma_{1}$.

2.$\gamma$ is a simple closed curve if $\gamma_{s}=\gamma_{t}$ implies
either $s=t$ or $\left\{ s,t\right\} =\left\{ 0,1\right\} $. 

3.$\gamma$ is a simple curve if $\gamma_{s}=\gamma_{t}$ implies
$s=t$. 
\end{defn}
Let $\gamma:\left[0,1\right]\rightarrow\mathbb{R}^{2}$ be a continuous
function. Let $z\in\mathbb{R}^{2}\backslash\gamma\left[0,1\right]$.
Then 
\[
g_{z}^{\gamma}\left(s\right):=\frac{\gamma_{s}-z}{\left\Vert \gamma_{s}-z\right\Vert }
\]
defines a function $\left[0,1\right]\rightarrow\mathbb{S}^{1}$. 

Let $p:\mathbb{R}\rightarrow\mathbb{S}^{1}$, $p\left(x\right)=e^{ix}$
be a covering map for $\mathbb{S}^{1}$. Then there exists a continuous
lift $\tilde{g}_{z}^{\gamma}:\left[0,1\right]\rightarrow\mathbb{R}$
such that $p\circ\tilde{g}_{z}^{\gamma}=g_{z}^{\gamma}$. The winding
number of $\gamma$ will be defined in terms of $\tilde{g}_{s}\left(z\right)$
by the following lemma:
\begin{lem}
(\cite{Dr. Qian's complex analysis book}, Chapter 3 Lemma 1 and 2)\label{lem:winding def}Let
$\gamma:\left[0,1\right]\rightarrow\mathbb{R}^{2}$ be a continuous
closed curve, and $z\in\gamma\left[0,1\right]$. Then the number 
\begin{equation}
\eta\left(\gamma,z\right):=\frac{1}{2\pi}\left(\tilde{g}_{z}^{\gamma}\left(1\right)-\tilde{g}_{z}^{\gamma}\left(0\right)\right)\label{eq:defn winding angle}
\end{equation}
depends only on $\gamma$ and $z$ but not on the lift $\tilde{g}_{z}^{\gamma}$.
Moreover, $\eta\left(\gamma,z\right)$ is an integer and is called
the winding number of $\gamma$ around the point $z$. \end{lem}
\begin{rem}
We may define the winding number for any $\gamma:\left[a,b\right]\rightarrow\mathbb{R}^{2}$
by simply replacing $0$ by $a$, $1$ by $b$ in the above definition. 
\end{rem}
The following theorem, which we shall need, is intuitively clear but
is highly non-trivial: 
\begin{thm}
(\cite{Munkres}, p404)\label{thm:Jordan winding}Let $\gamma:\left[0,1\right]\rightarrow\mathbb{R}^{2}$
be a simple closed curve. Let Int($\gamma$) and Ext($\gamma$) be
its interior and exterior respectively. Then $\eta\left(\gamma,z\right)=0$
for all $z\in$Ext($\gamma$). Moreover, either $\eta\left(\gamma,z\right)=1$
for all $z\in Int(\gamma)$ or $\eta\left(\gamma,z\right)=-1$ for
all $z\in Int\left(\gamma\right)$. $\gamma$ is called positively
oriented if $\eta\left(\gamma,z\right)=1$ and negatively oriented
otherwise. 
\end{thm}
A key tool in our proof of Theorem \ref{thm:winding 1} is the following
Green's theorem for paths with bounded total variations.
\begin{thm}
(\cite{Greens p equal one} and \cite{green p11})Let $\gamma=\left(\gamma^{\left(1\right)},\gamma^{\left(2\right)}\right):\left[0,T\right]\rightarrow\mathbb{R}^{2}$
be a closed curve with bounded total variation. Let $f,g:\mathbb{R}^{2}\rightarrow\mathbb{R}$
have continuous partial derivatives in both variables. Then 
\begin{equation}
\int_{\mathbb{R}^{2}}\left(\partial_{x}f\left(x,y\right)+\partial_{y}g\left(x,y\right)\right)\eta\left(\gamma,\left(x,y\right)\right)\mathrm{d}x\mathrm{d}y=\int_{\gamma}f\mathrm{d}\gamma_{s}^{\left(2\right)}-g\mathrm{d}\gamma_{s}^{\left(1\right)}.\label{eq:green1-2}
\end{equation}
and 
\begin{equation}
\left\Vert \eta\left(\gamma,\cdot\right)\right\Vert _{L^{2}}\leq\frac{1}{\sqrt{4\pi}}\left\Vert \gamma\right\Vert _{\mathcal{V}^{1}\left(\left[0,T\right],\mathbb{R}^{2}\right)}\label{eq:isoperimetric}
\end{equation}
where the equality in (\ref{eq:isoperimetric}) holds if and only
if there exists $\left(x,y\right)\in\mathbb{R}^{2}$, $n\in\mathbb{N}$
and $R>0$ such that $\gamma_{t}=\left(x+R\cos2\pi nt,x+R\sin2\pi nt\right)$.
\end{thm}
The $f\left(x,y\right)=x$, $g\left(x,y\right)=y$ case in (\ref{eq:green1-2})
was proved in \cite{Greens p equal one} and the proof for the general
case is essentially the same. New, complete proofs for (\ref{eq:green1-2})
were subsequently given by \cite{New proof green} and \cite{Yamthesis}. 

The second inequality is the well-known Banchoff-Pohl isoperimetric
inequality\cite{green p11}.

\section{Proof of Theorem \ref{thm:winding 1}}

Before we give a proof of Theorem \ref{thm:winding 1}, we would like
to first recall some elementary Lie algebra.

\subsection{Lyndon basis }

We shall briefly introduce the concept of Lyndon basis. For details,
readers are referred to \cite{Reutenauer}. Let $\mathcal{L}\left(\left\{ \mathbf{e}_{1},\mathbf{e}_{2}\right\} \right)$
be the set of Lie series generated by $\left\{ \mathbf{e}_{1},\mathbf{e}_{2}\right\} $
through the tensor product $\otimes$ and let $\mathcal{L}_{N}\left(\left\{ \mathbf{e}_{1},\mathbf{e}_{2}\right\} \right):=\pi_{N}\left(\mathcal{L}\left(\left\{ \mathbf{e}_{1},\mathbf{e}_{2}\right\} \right)\right)$.
We shall recall the definition of the Lyndon basis, which we used
to decompose $\pi_{N}\left(\log S\left(\gamma\right)_{0,1}\right)$
in Theorem \ref{thm:winding 1}. 

From here onwards, a \textit{word} will mean a monomial generated
by $\left\{ \mathbf{e}_{1},\mathbf{e}_{2}\right\} $ through $\otimes.$
The identity element with respect to $\otimes$ is the empty word
which will be denoted by $1$. We shall assign a lexicographical order
on the set of words by the following rule:
\begin{enumerate}
\item $\mathbf{e}_{1}<\mathbf{e}_{2}$.
\item If $\mathbf{v}=\mathbf{u}\otimes\mathbf{x}$ for some word $\mathbf{x}$,
then $\mathbf{u}<\mathbf{v}$. 
\item If $\mathbf{w}=\mathbf{u}\otimes\mathbf{e}_{1}\otimes\mathbf{x}$
and $\mathbf{w^{\prime}}=\mathbf{u}\otimes\mathbf{e}_{2}\otimes\mathbf{x^{\prime}}$
for words $\mathbf{u},\mathbf{x},\mathbf{x^{\prime}}$, then $\mathbf{w}<\mathbf{w^{\prime}}$.
\end{enumerate}
We say a word $\mathbf{w}$ is Lyndon if either $\mathbf{w}=\mathbf{e}_{1}$
or $\mathbf{w}=\mathbf{e}_{2}$ or for all $\mathbf{u}\neq1$,$\mathbf{v}\neq1$
such that $\mathbf{u}\otimes\mathbf{v}=\mathbf{w}$, we have $\mathbf{w}<\mathbf{v}$.
For each word $\mathbf{w}$, $\mathbf{w}\neq\mathbf{e}_{1},\mathbf{e}_{2}$,
if $\mathbf{v}$ is the smallest non-empty Lyndon word such that $\mathbf{w}=\mathbf{u}\otimes\mathbf{v}$
for some non-empty word $\mathbf{u}$, then we say $\mathbf{w}=\mathbf{u}\otimes\mathbf{v}$
is the standard factorisation of a Lyndon word $\mathbf{w}$.
\begin{example}
\label{Lyndon words level 4}The Lyndon words of degree less than
or equal to $4$ generated by $\left\{ \mathbf{e}_{1},\mathbf{e}_{2}\right\} $
are 
\[
\mathbf{e}_{1}<\mathbf{e}_{1}^{\otimes3}\otimes\mathbf{e}_{2}<\mathbf{e}_{1}^{\otimes2}\otimes\mathbf{e}_{2}<\mathbf{e}_{1}^{\otimes2}\otimes\mathbf{e}_{2}^{\otimes2}<\mathbf{e}_{1}\otimes\mathbf{e}_{2}<\mathbf{e}_{1}\otimes\mathbf{e}_{2}^{\otimes2}<\mathbf{e}_{1}\otimes\mathbf{e}_{2}^{\otimes3}<\mathbf{e}_{2}.
\]

\end{example}
For each Lyndon word, we can associate a corresponding Lyndon element
$\mathcal{P}_{\mathbf{w}}$ inductively by $\mathcal{P}_{\mathbf{e}_{1}}=\mathbf{e}_{1},\;\mathcal{P}_{\mathbf{e}_{2}}=\mathbf{e}_{2}$
and $\mathcal{P}_{\mathbf{w}}=\left[\mathcal{P}_{\mathbf{u}},\mathcal{P}_{\mathbf{v}}\right]$
if $\mathbf{w}=\mathbf{u}\mathbf{v}$ is the standard factorisation.
By Theorem 4.9 and Theorem 5.1 in \cite{Reutenauer}, the set 
\[
\left\{ \mathcal{P}_{\mathbf{w}}:\mathbf{w}\;\mbox{is a Lyndon word}\right\} 
\]
forms a basis of $\mathcal{L}\left(\left\{ \mathbf{e}_{1},\mathbf{e}_{2}\right\} \right)$. 

We shall now state a few key properties of the Lyndon words which
we shall use.
\begin{lem}
\label{lem:Lyndon words}1. (\cite{Reutenauer}, (5.1.2))Let $\mathbf{u}<\mathbf{v}$
be two Lyndon words. Then $\mathbf{u}\otimes\mathbf{v}$ is also a
Lyndon word. 

2.(\cite{Reutenauer}, Theorem 5.1) Let $n\in\mathbb{N}$. Let $\mathbf{w}$
be a Lyndon word such that $\mathbf{w}=\mathbf{l_{1}}\ldots\mathbf{l_{n}}$,
where $\mathbf{l_{1}}\geq\mathbf{l_{2}}\geq\ldots\geq\mathbf{l_{n}}$
are Lyndon words. Then $\mathcal{P}_{\mathbf{w}}=\mathbf{w}+h.o.t$
where $h.o.t$ is a linear combination over $\mathbb{Z}$ of words
strictly greater than $\mathbf{w}$. 
\end{lem}
From which it follows easily that: 
\begin{cor}
$\mathbf{e}_{1}^{\otimes n}\mathbf{e}_{2}^{\otimes k}$ is a Lyndon
word for all $n>0$ and $k>0$. \end{cor}
\begin{proof}
Iterative use of 1. in Lemma \ref{lem:Lyndon words}. 
\end{proof}

\subsection{Proof of Theorem \ref{thm:winding 1}}

We first need a technical lemma which controls the $L^{q}$ norm of
the winding number. 
\begin{lem}
\label{lem:Winding integrability}Let $1\leq p<2$. Then for all $q<\frac{2}{p}$,
there exists $C_{p,q}>0$ such that for all paths $\gamma:\left[0,1\right]\rightarrow\mathbb{R}^{2}$
with finite $p$ variation,.
\[
\left\Vert \eta\left(\gamma,\cdot\right)\right\Vert _{L^{q}}\leq C_{p,q}\max\left(\left\Vert \gamma\right\Vert _{p},\left\Vert \gamma\right\Vert _{p}^{p}\right).
\]
 \end{lem}
\begin{proof}
First consider the case when $\gamma$ has finite total variation.
For such paths, we have the integral representation 
\[
\eta\left(\gamma,\left(x,y\right)\right)=\frac{1}{2\pi}\int_{0}^{1}\frac{\left(x_{s}-x\right)\mathrm{d}y_{s}-\left(y_{s}-y\right)\mathrm{d}x_{s}}{\left(x_{s}-x\right)^{2}+\left(y_{s}-y\right)^{2}}.
\]
Let $f\in L^{q}\left(\mathbb{R}^{2}\right)$, where $q>\frac{2}{2-p}$.
Consider the map $f\rightarrow\int_{\mathbb{R}^{2}}f\left(z\right)\eta\left(\gamma,z\right)\mathrm{d}z$.
By an interchange of integral, we have 
\begin{eqnarray*}
 &  & \int_{\mathbb{R}^{2}}f\left(x,y\right)\eta\left(\gamma,\left(x,y\right)\right)\mathrm{d}x\mathrm{d}y\\
 & = & \frac{1}{2\pi}\left(\mathbf{e}_{2}^{*}\otimes\mathbf{e}_{1}^{*}-\mathbf{e}_{1}^{*}\otimes\mathbf{e}_{2}^{*}\right)\int_{0}^{1}\left(\int_{\mathbb{R}^{2}}\frac{\gamma_{s}-\left(x,y\right)}{\left|\gamma_{s}-\left(x,y\right)\right|^{2}}f\left(x,y\right)\mathrm{d}x\mathrm{d}y\right)\otimes\mathrm{d}\gamma_{s}.
\end{eqnarray*}

The quasi-potential operator $T$ defined by 
\[
T\left(f\right)\left(z\right):=\int_{\mathbb{R}^{2}}\frac{z-\left(x,y\right)}{\left|z-\left(x,y\right)\right|^{2}}f\left(x,y\right)\mathrm{d}x\mathrm{d}y
\]
is a bounded linear operator from $L^{q}\left(\mathbb{R}^{2}\right)$
to $\mbox{Lip}\left(1-\frac{2}{q}\right)$ (See Theorem 3.7.1 in \cite{Morrey's book}). 

Note that as $q>\frac{2}{2-p}$, $2-\frac{2}{q}>p$. Therefore, 
\begin{eqnarray*}
\left|\int_{0}^{1}\left(Tf\right)\left(\gamma_{s}\right)\otimes\mathrm{d}\gamma_{s}\right| & \leq & C_{p,q}\left\Vert \left(Tf\right)\left(\cdot\right)\right\Vert _{\mbox{Lip}\left(2-\frac{2}{q}\right)}\max\left(\left\Vert \gamma\right\Vert _{p}^{p},\left\Vert \gamma\right\Vert _{p}\right).
\end{eqnarray*}

Therefore, the map 
\[
f\rightarrow\int_{\mathbb{R}^{2}}f\left(x,y\right)\eta\left(\gamma,\left(x,y\right)\right)\mathrm{d}x\mathrm{d}y
\]
is a bounded linear functional on $L^{q}$ and 

\[
\left|\int_{\mathbb{R}^{2}}f\left(x,y\right)\eta\left(\gamma,\left(x,y\right)\right)\mathrm{d}x\mathrm{d}y\right|\leq C_{p,q}\left\Vert f\right\Vert _{q}\max\left(\left\Vert \gamma\right\Vert _{p}^{p},\left\Vert \gamma\right\Vert _{p}\right).
\]

This means for all paths $\gamma$ with bounded total variation, and
all $q>\frac{2}{2-p}$, or $q^{\prime}<\frac{2}{p}$, 
\[
\left\Vert \eta\left(\gamma,\cdot\right)\right\Vert _{L^{q^{\prime}}}\leq C_{p,q}\max\left(\left\Vert \gamma\right\Vert _{p}^{p},\left\Vert \gamma\right\Vert _{p}\right)
\]
where $C_{p,q}$ is a constant independent of $\gamma$. 

Let $\gamma$ now be a path with finite $p$ variation, where $p<2$.
Let $\mathcal{P}$ be any piecewise linear interpolation of $\gamma$.
Then 
\begin{eqnarray*}
\left\Vert \eta\left(\gamma^{\mathcal{P}},\cdot\right)\right\Vert _{L^{q^{\prime}}} & \leq & C_{p,q}\max\left(\left\Vert \gamma^{\mathcal{P}}\right\Vert _{p}^{p},\left\Vert \gamma^{\mathcal{P}}\right\Vert _{p}\right)\\
 & \leq & C_{p,q}\max\left(\left\Vert \gamma\right\Vert _{p}^{p},\left\Vert \gamma\right\Vert _{p}\right).
\end{eqnarray*}

Let $\mathcal{P}_{n}$ be a sequence of partitions such that $\left\Vert \mathcal{P}_{n}\right\Vert \rightarrow0$
as $n\rightarrow\infty$. Then by Fatou's Lemma, 
\begin{eqnarray*}
\left\Vert \eta\left(\gamma,\cdot\right)\right\Vert _{L^{q^{\prime}}} & \leq & \lim_{n\rightarrow\infty}\left\Vert \eta\left(\gamma^{\mathcal{P}_{n}},\cdot\right)\right\Vert _{L^{q^{\prime}}}\\
 & \leq & C_{p,q}\max\left(\left\Vert \gamma\right\Vert _{p}^{p},\left\Vert \gamma\right\Vert _{p}\right).
\end{eqnarray*}

\end{proof}
A key idea in proving Theorem \ref{thm:winding 1} lies in the fact
that the coefficients of some Lyndon basis elements can be reduced
to a single line integral, as illustrated by the following lemma. 
\begin{lem}
\label{thm:main1}Let $1\leq p<2$. Let $\gamma:\left[0,1\right]\rightarrow\mathbb{R}^{2}$
be a continuous closed curve with finite $p$ variation. Let $\eta\left(\gamma,\left(x,y\right)\right)$
denote the winding number of $\gamma$ around $x\mathbf{e}_{1}+y\mathbf{e}_{2}$.

Then for all $n,k\geq0$, 
\begin{equation}
\mathbf{e}_{1}^{*\otimes\left(n+1\right)}\otimes\mathbf{e}_{2}^{*\otimes\left(k+1\right)}\left(S\left(\gamma\right)_{0,1}\right)=\frac{\left(-1\right)^{k}}{n!k!}\int_{\mathbb{R}^{2}}x^{n}y^{k}\eta\left(\gamma-\gamma_{0},\left(x,y\right)\right)\mathrm{d}x\mathrm{d}y.\label{eq:Green's main}
\end{equation}
\end{lem}
\begin{proof}
We first prove the lemma for paths with bounded total variation.

Let $\gamma^{\left(1\right)}$and $\gamma^{\left(2\right)}$ be the
first and second coordinate components of $\gamma$ respectively.

Recall that for all $n,k\geq0$, 
\[
\mathbf{e}_{1}^{*\otimes\left(n+1\right)}\otimes\mathbf{e}_{2}^{*\otimes\left(k+1\right)}\left(S\left(\gamma\right)_{0,1}\right)=\int_{\triangle_{n+k+2}\left(0,1\right)}\mathrm{d}\gamma_{s_{1}}^{\left(1\right)}\ldots\mathrm{d}\gamma_{s_{n+1}}^{\left(1\right)}\mathrm{d}\gamma_{s_{n+2}}^{\left(2\right)}\ldots\mathrm{d}\gamma_{s_{n+k+2}}^{\left(2\right)}.
\]

The key idea here is to integrate with respect to $\gamma^{\left(1\right)}$s
first and then integrate the $\gamma^{\left(2\right)}$s. For all
$n,k\geq0$, 
\begin{eqnarray*}
 &  & \mathbf{e}_{1}^{*\otimes\left(n+1\right)}\otimes\mathbf{e}_{2}^{*\otimes\left(k+1\right)}\left(S\left(\gamma\right)_{0,1}\right)\\
 & = & \int...\int_{0<t_{1}<..<t_{n+1}<s_{1}<...<s_{k+1}<1}\mathrm{d}\gamma_{t_{1}}^{\left(1\right)}...\mathrm{d}\gamma_{t_{n+1}}^{\left(1\right)}\mathrm{d}\gamma_{s_{1}}^{\left(2\right)}...\mathrm{d}\gamma_{s_{k+1}}^{\left(2\right)}\\
 & = & \int_{0<s_{1}<...<s_{k+1}<1}\frac{1}{n!}\left(\gamma_{s_{1}}^{\left(1\right)}-\gamma_{0}^{\left(1\right)}\right)^{n+1}\mathrm{d}\gamma_{s_{1}}^{\left(2\right)}...\mathrm{d}\gamma_{s_{k+1}}^{\left(2\right)}\\
 & = & \int_{0}^{1}\int_{s_{1}}^{1}...\int_{s_{k-1}}^{1}\int_{s_{k}}^{1}\frac{1}{n!}\left(\gamma_{s_{1}}^{\left(1\right)}-\gamma_{0}^{\left(1\right)}\right)^{n+1}\mathrm{d}\gamma_{s_{k+1}}^{\left(2\right)}...\mathrm{d}\gamma_{s_{1}}^{\left(2\right)}\;\mbox{by Fubini's theorem}\\
 & = & \frac{1}{\left(n+1\right)!}\frac{1}{k!}\int_{0}^{1}\left(\gamma_{s_{1}}^{\left(1\right)}-\gamma_{0}^{\left(1\right)}\right)^{n+1}\left(\gamma_{1}^{\left(2\right)}-\gamma_{s_{1}}^{\left(2\right)}\right)^{k}\mathrm{d}\gamma_{s_{1}}^{\left(2\right)}\\
 & = & \frac{1}{n!}\frac{1}{k!}\int_{\mathbb{R}^{2}}\left(x-\gamma_{0}^{\left(1\right)}\right)^{n}\left(\gamma_{1}^{\left(2\right)}-y\right)^{k}\eta\left(\gamma,\left(x,y\right)\right)\mathrm{d}x\mathrm{d}y\;\mbox{by }(\ref{eq:green1-2})\\
 & = & \frac{\left(-1\right)^{k}}{n!k!}\int_{\mathbb{R}^{2}}x^{n}y^{k}\eta\left(\gamma-\gamma_{0},\left(x,y\right)\right)\mathrm{d}x\mathrm{d}y.
\end{eqnarray*}
where in the last two steps we have used the fact that $\gamma$ is
a closed curve. 

Now for $\gamma$ with finite $p$ variation, for each $N\in\mathbb{N}$,
let $\mathcal{P}_{N}$ denote a sequence of partitions of $\left[0,1\right]$
such that $\left\Vert \mathcal{P}_{N}\right\Vert \rightarrow0$ as
$N\rightarrow\infty$. Then by what we just proved, 
\begin{equation}
\mathbf{e}_{1}^{*\otimes\left(n+1\right)}\otimes\mathbf{e}_{2}^{*\otimes\left(k+1\right)}\left(S\left(\gamma^{\mathcal{P}_{n}}\right)_{0,1}\right)=\frac{\left(-1\right)^{k}}{n!k!}\int_{\mathbb{R}^{2}}x^{n}y^{k}\eta\left(\gamma^{\mathcal{P}_{N}}-\gamma_{0},\left(x,y\right)\right)\mathrm{d}x\mathrm{d}y.\label{eq:Green's main proof 2}
\end{equation}

We will now take limit as $N\rightarrow\infty$. The left hand side
of (\ref{eq:Green's main proof 2}) converges to $S\left(\gamma\right)_{0,1}$
by Lemma \ref{lem:approx3}. 

To show that the right hand side of (\ref{eq:Green's main proof 2})
converges to 
\[
\frac{\left(-1\right)^{k}}{n!k!}\int_{\mathbb{R}^{2}}x^{n}y^{k}\eta\left(\gamma-\gamma_{0},\left(x,y\right)\right)\mathrm{d}x\mathrm{d}y
\]
, note that by Lemma \ref{lem:Winding integrability}, if we take
$1<q<\frac{2}{p}$, 
\[
\left\Vert \eta\left(\gamma^{\mathcal{P}_{N}},\cdot\right)\right\Vert _{L^{q}}\leq C_{p,q}\max\left(\left\Vert \gamma^{\mathcal{P}_{N}}\right\Vert _{p},\left\Vert \gamma^{\mathcal{P}_{N}}\right\Vert _{p}^{p}\right)\leq C_{p,q}\max\left(\left\Vert \gamma\right\Vert _{p},\left\Vert \gamma\right\Vert _{p}^{p}\right)
\]
and the convergence follows from $L^{q}$ convergence theorems. 
\end{proof}
\textit{Update} : (Jul 2014) Instead of proving the integrability
of winding number, we could also have used the generalised Green's
theorem in Proposition 4.6 in \cite{Zust11}. We are grateful to R.
Züst for pointing this out to us. 

We will now give a proof of Theorem \ref{thm:winding 1}.

\selectlanguage{british}%
\begin{proof}[Proof of Theorem 1]Let $n,k\geq0$ and $N\geq n+k+2$.
If we equip the alphabet $\left\{ \mathbf{e}_{1},\mathbf{e}_{2}\right\} $
with the ordering $\mathbf{e}_{1}<\mathbf{e}_{2}$, then by Lemma
\ref{lem:Lyndon words}, $\mathbf{e}_{1}{}^{\otimes n}\otimes\mathbf{e}_{2}{}^{\otimes k}$
is a Lyndon word as defined in section 3.1. Let $\mathcal{P}_{\mathbf{e}_{1}^{\otimes n}\otimes\mathbf{e}_{2}^{\otimes k}}$
denote the corresponding Lyndon element. By Lemma \ref{thm:main1},
it suffices to prove that for all $n,k\geq0$ and $N\geq n+k+2$,
\begin{align*}
 & \mathcal{P}_{\mathbf{e}_{1}{}^{\otimes n+1}\otimes\mathbf{e}_{2}{}^{\otimes k+1}}^{*}\left(\log S_{N}\left(\gamma\right)_{0,1}\right)=\mathbf{e}_{1}^{*\otimes n+1}\otimes\mathbf{e}_{2}^{*\otimes k+1}\left(S_{N}\left(\gamma\right)_{0,1}\right).
\end{align*}

We will first prove that for closed curve $\gamma$, for all $n\geq0$,$k\geq0$,
\begin{equation}
\mathbf{e}_{1}^{*\otimes n+1}\otimes\mathbf{e}_{2}^{*\otimes k+1}\left(\left(\log S\left(\gamma\right)_{0,1}\right)^{\otimes j}\right)=0\label{eq:too much change}
\end{equation}
for $j\geq2$. 

First note that as $\gamma$ is a closed curve 
\begin{equation}
\mathbf{e}_{1}^{*}\left(\log S\left(\gamma\right)_{0,1}\right)=\mathbf{e}_{2}^{*}\left(\log S\left(\gamma\right)_{0,1}\right)=0.\label{eq:no first erm}
\end{equation}

If we denote the coefficient of a word $w$ in a polynomial $\mathcal{P}$
by $\left(\mathcal{P},w\right)$, then for all $n,k\geq0$,

\begin{eqnarray*}
 &  & \mathbf{e}_{1}^{*\otimes n+1}\otimes\mathbf{e}_{2}^{*\otimes k+1}\left(\left(\log S\left(\gamma\right)_{0,1}\right)^{\otimes j}\right)\\
 & = & \sum_{w_{1}\ldots w_{j}=\mathbf{e}_{1}^{\otimes n+1}\otimes\mathbf{e}_{2}^{\otimes k+1}}\left(\pi^{\left(n+k+2\right)}\left(\log S\left(\gamma\right)_{0,1}\right),w_{1}\right)\ldots\left(\pi^{\left(n+k+2\right)}\left(\log S\left(\gamma\right)_{0,1}\right),w_{j}\right).
\end{eqnarray*}

For each ordered collection of words $w_{1},\ldots,w_{j}$ satisfying
$w_{1},\ldots,w_{j}=\mathbf{e}_{1}^{\otimes n+1}\otimes\mathbf{e}_{2}^{\otimes k+1}$,
then at least one of $w_{1},\ldots,w_{j}$ will be of the form $\mathbf{e}_{i}^{\otimes l}$
where $i=1$ or $2$ for some $l\geq1$. Denote this word by $w^{\prime}$.
As $\pi^{\left(n+k+2\right)}\left(\log S\left(\gamma\right)_{0,1}\right)$
is a Lie polynomial and the first degree term of $\log S\left(\gamma\right)_{0,1}$is
zero (see (\ref{eq:no first erm})), 
\[
\left(\pi^{\left(n+k+2\right)}\left(\log S\left(\gamma\right)_{0,1}\right),w^{\prime}\right)=0
\]
which proves (\ref{eq:too much change}). 

Therefore, for all $n,k\geq0$, 
\begin{eqnarray*}
\mathbf{e}_{1}^{*\otimes n+1}\otimes\mathbf{e}_{2}^{*\otimes k+1}\left(S\left(\gamma\right)_{0,1}\right) & = & \mathbf{e}_{1}^{*\otimes n+1}\otimes\mathbf{e}_{2}^{*\otimes k+1}\left(\log S\left(\gamma\right)_{0,1}\right).
\end{eqnarray*}

Suppose we now expand $\pi^{\left(n+k+2\right)}\left(\log S\left(\gamma\right)_{0,1}\right)$
in terms of Lyndon words $\sum_{\mbox{Lyndon words }\mathbf{h}}\mathcal{P}_{\mathbf{h}}^{*}\circ\pi^{\left(n+k+2\right)}\left(\log S\left(\gamma\right)_{0,1}\right)\mathcal{P}_{\mathbf{h}}$,
then for all $n,k\geq0$,

\begin{eqnarray*}
\mathbf{e}_{1}^{*\otimes n+1}\otimes\mathbf{e}_{2}^{*\otimes k+1}\left(S\left(\gamma\right)_{0,1}\right) & = & \sum_{\mbox{Lyndon words }\mathbf{h}}\mathcal{P}_{\mathbf{h}}^{*}\circ\pi_{N}\left(\log S\left(\gamma\right)_{0,1}\right)\mathbf{e}_{1}^{*\otimes n+1}\otimes\mathbf{e}_{2}^{*\otimes k+1}\left(\mathcal{P}_{\mathbf{h}}\right).
\end{eqnarray*}

By definition, $\mathbf{e}_{1}^{*\otimes n+1}\otimes\mathbf{e}_{2}^{*\otimes k+1}\left(\mathcal{P}_{\mathbf{h}}\right)$
will be non-zero only if the word $\mathbf{h}$ contains $n+1$ letters
$\mathbf{e}_{1}$ and $k+1$ letters $\mathbf{e}_{2}$. If $\mathbf{h}$
contains $n+1$ $\mathbf{e}_{1}$s and $k+1$ $\mathbf{e}_{2}$s,
then by Lemma \ref{lem:Lyndon words}, 
\begin{equation}
\mathcal{P}_{\mathbf{h}}=\mathbf{h}+\mathbb{Z}-\mbox{ linear combination of }\mbox{words greater than }\mathbf{h}.\label{eq:Lyndon propert}
\end{equation}

However, $\mathbf{e}_{1}^{\otimes n+1}\otimes\mathbf{e}_{2}^{\otimes k+1}$
is the smallest word amongst all words with $n+1$ $\mathbf{e}_{1}$s
and $k+1$ $\mathbf{e}_{2}$s. Therefore, if $\mathbf{h}\neq\mathbf{e}_{1}^{*\otimes n+1}\otimes\mathbf{e}_{2}^{*\otimes k+1}$,
then the right hand side of (\ref{eq:Lyndon propert}) will only contain
words strictly greater than $\mathbf{e}_{1}^{\otimes n+1}\otimes\mathbf{e}_{2}^{\otimes k+1}$
and in particular will not contain the word $\mathbf{e}_{1}^{\otimes n+1}\otimes\mathbf{e}_{2}^{\otimes k+1}$.
Therefore, for all $n,k\geq0$ 
\[
\mathbf{e}_{1}^{*\otimes n+1}\otimes\mathbf{e}_{2}^{*\otimes k+1}\left(\mathcal{P}_{\mathbf{h}}\right)=0\;\mbox{if }\mathbf{h}\neq\mathbf{e}_{1}^{\otimes n+1}\otimes\mathbf{e}_{2}^{\otimes k+1}.
\]

Therefore, for all $n,k\geq0$, 
\[
\mathbf{e}_{1}^{*\otimes n+1}\otimes\mathbf{e}_{2}^{*\otimes k+1}\left(S\left(\gamma\right)_{0,1}\right)=\mathcal{P}_{\mathbf{e}_{1}^{\otimes n+1}\otimes\mathbf{e}_{2}^{\otimes k+1}}^{*}\circ\pi_{N}\left(\log S\left(\gamma\right)_{0,1}\right).
\]
\end{proof}

\selectlanguage{english}%
We now prove Corollary \ref{cor:winding 2}. 

\begin{proof}[Proof of Corollary 2]

\selectlanguage{british}%
In Example \ref{Lyndon words level 4}, we listed the Lyndon words
of length less than or equal to $4$. The corresponding Lyndon elements
for the free Lie algebra generated by the alphabet $\left\{ \mathbf{e}_{1},\mathbf{e}_{2}\right\} $
is 
\begin{align}
 & \begin{array}{c}
\mathbf{e}_{1},\left[\mathbf{e}_{1},\left[\mathbf{e}_{1},\left[\mathbf{e}_{1},\mathbf{e}_{2}\right]\right]\right],\left[\mathbf{e}_{1},\left[\mathbf{e}_{1},\mathbf{e}_{2}\right]\right],\left[\mathbf{e}_{1},\left[\left[\mathbf{e}_{1},\mathbf{e}_{2}\right],\mathbf{e}_{2}\right]\right],\left[\mathbf{e}_{1},\mathbf{e}_{2}\right]\\
,\left[\left[\mathbf{e}_{1},\mathbf{e}_{2}\right],\mathbf{e}_{2}\right],\left[\left[\left[\mathbf{e}_{1},\mathbf{e}_{2}\right],\mathbf{e}_{2}\right],\mathbf{e}_{2}\right],\mathbf{e}_{2}
\end{array}\label{eq: Hall basis 4}
\end{align}

To prove Corollary \foreignlanguage{english}{\ref{cor:winding 2},
it is sufficient to express, for each of the above Lyndon elements
$f$, the coefficient of $f$ in $\log S\left(\gamma\right)_{0,1}$
in terms of the winding number of $\gamma$. }

\selectlanguage{english}%
As $\gamma$ is a closed curve, $\mathbf{e}_{i}^{*}\left(\log\left(S\left(\gamma\right)_{0,1}\right)\right)=0$
for $i=1,2$. 

By Theorem \ref{thm:winding 1}, 
\begin{align}
\left[\mathbf{e}_{1},\mathbf{e}_{2}\right]^{*}\circ\pi_{4}\left(\log S\left(\gamma\right)_{0,1}\right) & =\int_{\mathbb{R}^{2}}\eta\left(\gamma-\gamma_{0},\left(x,y\right)\right)\mathrm{d}x\mathrm{d}y\nonumber \\
\left[\mathbf{e}_{1},\left[\mathbf{e}_{1},\mathbf{e}_{2}\right]\right]^{*}\circ\pi_{4}\left(\log S\left(\gamma\right)_{0,1}\right) & =\int_{\mathbb{R}^{2}}x\eta\left(\gamma-\gamma_{0},\left(x,y\right)\right)\mathrm{d}x\mathrm{d}y\nonumber \\
\left[\left[\mathbf{e_{1},e_{2}}\right],\mathbf{e}_{2}\right]^{*}\circ\pi_{4}\left(\log S\left(\gamma\right)_{0,1}\right) & =-\int_{\mathbb{R}^{2}}y\eta\left(\gamma-\gamma_{0},\left(x,y\right)\right)\mathrm{d}x\mathrm{d}y\nonumber \\
\left[\mathbf{e}_{1},\left[\mathbf{e}_{1},\left[\mathbf{e_{1},e_{2}}\right]\right]\right]^{*}\circ\pi_{4}\left(\log S\left(\gamma\right)_{0,1}\right) & =\frac{1}{2}\int_{\mathbb{R}^{2}}x^{2}\eta\left(\gamma-\gamma_{0},\left(x,y\right)\right)\mathrm{d}x\mathrm{d}y\nonumber \\
\left[\mathbf{e}_{1},\left[\left[\mathbf{e_{1},e_{2}}\right],\mathbf{e}_{2}\right]\right]^{*}\circ\pi_{4}\left(\log S\left(\gamma\right)_{0,1}\right) & =-\int_{\mathbb{R}^{2}}xy\eta\left(\gamma-\gamma_{0},\left(x,y\right)\right)\mathrm{d}x\mathrm{d}y\nonumber \\
\left[\left[\left[\mathbf{e_{1},e_{2}}\right],\mathbf{e}_{2}\right],\mathbf{e}_{2}\right]^{*}\circ\pi_{4}\left(\log S\left(\gamma\right)_{0,1}\right) & =\frac{1}{2}\int_{\mathbb{R}^{2}}y^{2}\eta\left(\gamma-\gamma_{0},\left(x,y\right)\right)\mathrm{d}x\mathrm{d}y.\label{eq:4th term}
\end{align}
\end{proof}

\subsection{Sharpness of Corollary \ref{cor:winding 2}}

\selectlanguage{british}%
The purpose of this section is to prove the following sharpness compliment
to Corollary \foreignlanguage{english}{\ref{cor:winding 2}. }
\begin{prop}
\label{prop:sharpness}There exist two paths $\gamma,\tilde{\gamma}$
such that the winding numbers of $\gamma$ and $\tilde{\gamma}$ around
every point are equal, but the fifth term of their signature differs. \end{prop}
\begin{proof}
Let $\mathbf{e}_{i}$ denote the path $t\rightarrow t\mathbf{e}_{i}$,
$t\in\left[0,1\right]$ and let 
\[
\gamma=\mathbf{\mathbf{e}_{1}\star\mathbf{e}_{2}\star-\mathbf{e}_{1}\star}\mathbf{-e_{2}\star-e_{1}\star-e_{2}\star e_{1}\star e_{2}}
\]
and 
\[
\tilde{\gamma}=\mathbf{-e_{1}\star-e_{2}\star e_{1}\star e_{2}}\star\mathbf{e}_{1}\star\mathbf{e}_{2}\star-\mathbf{e}_{1}\star\mathbf{-e}_{2}.
\]
where $\star$ denotes the concatenation operation on paths. 

By Theorem \ref{thm:Jordan winding} and the additivity of the winding
number with respect to the concatenation product, 
\[
\eta\left(\gamma,\left(x,y\right)\right)=1_{\left[0,1\right]\times\left[0,1\right]\cup\left[-1,0\right]\times\left[-1,0\right]}\left(x,y\right)=\eta\left(\tilde{\gamma},\dot{\left(x,y\right)}\right).
\]

\selectlanguage{english}%
By a direct calculation, we see that the signature of $\mathbf{e}_{i}$
is 
\[
e^{\mathbf{e}_{i}}.
\]

Therefore, by Chen's identity, 

\selectlanguage{british}%
\begin{eqnarray}
S\left(\gamma\right)_{0,1} & = & e^{\mathbf{e}_{1}}e^{\mathbf{e}_{2}}e^{-\mathbf{e}_{1}}e^{-\mathbf{e}_{2}}e^{-\mathbf{e}_{1}}e^{-\mathbf{e}_{2}}e^{\mathbf{e}_{1}}e^{\mathbf{e}_{2}}\label{eq:s gamma}
\end{eqnarray}
and 
\begin{equation}
S\left(\tilde{\gamma}\right)_{0,1}=e^{-\mathbf{e}_{1}}e^{-\mathbf{e}_{2}}e^{\mathbf{e}_{1}}e^{\mathbf{e}_{2}}e^{\mathbf{e}_{1}}e^{\mathbf{e}_{2}}e^{-\mathbf{e}_{1}}e^{-\mathbf{e}_{2}}.\label{eq:s gamma h}
\end{equation}

We claim that 
\[
\mathbf{e}_{1}^{*}\otimes\mathbf{e}_{2}^{*}\otimes\mathbf{e}_{1}^{*}\otimes\mathbf{e}_{2}^{*}\otimes\mathbf{e}_{1}^{*}\left(S\left(\gamma\right)_{0,1}\right)=1
\]
and 
\[
\mathbf{e}_{1}^{*}\otimes\mathbf{e}_{2}^{*}\otimes\mathbf{e}_{1}^{*}\otimes\mathbf{e}_{2}^{*}\otimes\mathbf{e}_{1}^{*}\left(S\left(\tilde{\gamma}\right)_{0,1}\right)=-1.
\]

Note that the word $\mathbf{e}_{1}\otimes\mathbf{e}_{2}\otimes\mathbf{e}_{1}\otimes\mathbf{e}_{2}\otimes\mathbf{e}_{1}$
is ``square-free'', i.e. none of the letters in the word is identical
to the letter on its immediate left or right. This means the contribution
to the value of both 
\[
\mathbf{e}_{1}^{*}\otimes\mathbf{e}_{2}^{*}\otimes\mathbf{e}_{1}^{*}\otimes\mathbf{e}_{2}^{*}\otimes\mathbf{e}_{1}^{*}\left(S\left(\gamma\right)_{0,1}\right)
\]
and 
\[
\mathbf{e}_{1}^{*}\otimes\mathbf{e}_{2}^{*}\otimes\mathbf{e}_{1}^{*}\otimes\mathbf{e}_{2}^{*}\otimes\mathbf{e}_{1}^{*}\left(S\left(\tilde{\gamma}\right)_{0,1}\right)
\]
only comes from the first order term in exponentials in (\ref{eq:s gamma})
and (\ref{eq:s gamma h}). For both, the contribution can only come
in one of the following five combinations: 

Combination 1. $1^{st},2^{nd},3^{rd},4^{th},5^{th}$ exponentials. 

Combination 2. $1^{st},2^{nd},3^{rd},4^{th},7^{th}$ exponentials. 

Combination 3. $1^{st},2^{nd},3^{rd},6^{th},7^{th}$ exponentials.

Combination 4. $1^{st},2^{nd},5^{rd},6^{th},7^{th}$ exponentials.

Combination 5. $1^{st},4^{nd},5^{rd},6^{th},7^{th}$ exponentials.

For $S\left(\gamma\right)_{0,1}$, the contributions from Combination
1 and Combination 5 is $-1$, while the contribution from Combination
$2$ to $4$ is $1$. Therefore, 
\begin{eqnarray*}
 &  & \mathbf{e}_{1}^{*}\otimes\mathbf{e}_{2}^{*}\otimes\mathbf{e}_{1}^{*}\otimes\mathbf{e}_{2}^{*}\otimes\mathbf{e}_{1}^{*}\left(S\left(\gamma\right)_{0,1}\right)\\
 & = & -1+1+1+1-1\\
 & = & 1.
\end{eqnarray*}

For $S\left(\gamma^{\prime}\right)_{0,1}$, the contributions from
Combination 1 and Combination 5 is $1$, while the contribution from
Combination $2-4$ is $-1$. Therefore,
\begin{eqnarray*}
 &  & \mathbf{e}_{1}^{*}\otimes\mathbf{e}_{2}^{*}\otimes\mathbf{e}_{1}^{*}\otimes\mathbf{e}_{2}^{*}\otimes\mathbf{e}_{1}^{*}\left(S\left(\tilde{\gamma}\right)_{0,1}\right)\\
 & = & 1-1-1-1+1\\
 & = & -1.
\end{eqnarray*}

\end{proof}

\subsection{``tree-like'' paths and winding number}
\selectlanguage{english}%
\begin{prop}
\label{lem:tree-like winding}Let $1\leq p<2$. If a two dimensional
path $\gamma$ with finite $p$-variation has trivial signature then
$\gamma$ is closed and has winding number zero around all points
$\left(x,y\right)$ in $\mathbb{R}^{2}\backslash\gamma\left[0,1\right]$. \end{prop}
\begin{proof}
As the first term of the signature of $\gamma$ is zero, we have 
\[
\int_{0}^{1}\mathrm{d}\gamma=\gamma_{1}-\gamma_{0}=0.
\]

By Theorem \ref{thm:winding 1}, 
\[
\int_{\mathbb{R}^{2}}\frac{x^{n}y^{k}}{n!k!}\eta\left(\gamma-\gamma_{0},\left(x,y\right)\right)\mathrm{d}x\mathrm{d}y=0
\]
for all $n,k\geq0$. Therefore, 
\[
\int_{\mathbb{R}^{2}}e^{\lambda_{1}ix+\lambda_{2}iy}\eta\left(\gamma-\gamma_{0},\left(x,y\right)\right)\mathrm{d}x\mathrm{d}y=0
\]
for all $\lambda_{1},\lambda_{2}\in\mathbb{R}$. As the function $\left(x,y\right)\rightarrow\eta\left(\gamma-\gamma_{0},\left(x,y\right)\right)$
lies in $L^{1}$, we have by the injectiveness of Fourier transform
on $L^{1}$ that 
\[
\eta\left(\gamma,\left(x,y\right)+\gamma_{0}\right)=\eta\left(\gamma-\gamma_{0},\left(x,y\right)\right)=0
\]
for all $\left(x,y\right)\in\mathbb{R}^{2}$ except a Lebesgue null
set. As the function $\left(x,y\right)\rightarrow\eta\left(\gamma,\left(x,y\right)+\gamma_{0}\right)$
is locally constant on $\mathbb{R}^{2}\backslash\gamma\left[0,1\right]$,
we have 
\[
\eta\left(\gamma-\gamma_{0},\left(x,y\right)\right)=0
\]
for all $\left(x,y\right)\in\mathbb{R}^{2}\backslash\gamma\left[0,1\right]$.\end{proof}
\begin{rem}
In \cite{tree like}, it was proved that the signature of a path with
bounded total variation is trivial if and only if the path is ``tree-like''
(See Definition 1.2 in \cite{tree like}). Therefore, Proposition
\ref{lem:tree-like winding} means that a planar tree-like path has
zero winding around every point in the plane. 
\end{rem}

\begin{rem}
The converse of Proposition \ref{lem:tree-like winding} is not true.
Let $\gamma$ and $\tilde{\gamma}$ be the paths defined in the proof
of Proposition \ref{prop:sharpness} and $\eta$ be the concatenation
of $\gamma$ and the reversal of $\tilde{\gamma}$. Then by the additivity
of winding number with respect to the concatenation product, $\eta$
has zero winding number around every point. As the signatures of $\gamma$
and $\tilde{\gamma}$ are different, we have by Chen's identity that
the signature of $\eta$ is not $\mathbf{1}$. Therefore, $\eta$
does not have trivial signature.
\end{rem}
\selectlanguage{british}%

\section{Uniqueness of signature}

\selectlanguage{english}%

\subsection{Proof of Theorem \ref{thm:jordan1}}

Let $p\geq1$. For elements $\gamma$ and $\tilde{\gamma}$ in $\mathcal{V}^{p}\left(\left[0,T_{2}\right],\mathbb{R}^{d}\right)$
and $\mathcal{V}^{p}\left(\left[0,T_{1}\right],\mathbb{R}^{d}\right)$,
define a concatenation product $\star$$:$$\mathcal{V}^{p}\left(\left[0,T_{2}\right],\mathbb{R}^{d}\right)$$\times$$\mathcal{V}^{p}\left(\left[0,T_{1}\right],\mathbb{R}^{d}\right)$
$\rightarrow$ $\mathcal{V}^{p}\left(\left[0,T_{1}+T_{2}\right],\mathbb{R}^{d}\right)$
by 
\begin{eqnarray*}
\gamma\star\tilde{\gamma}\left(u\right) & := & \gamma\left(u\right),\quad u\in\left[0,T_{1}\right],\\
\gamma\star\tilde{\gamma}\left(u\right) & := & \tilde{\gamma}\left(u-T_{1}\right)+\gamma\left(T_{1}\right)-\tilde{\gamma}\left(0\right),\quad u\in\left[T_{1},T_{1}+T_{2}\right]
\end{eqnarray*}

\selectlanguage{british}%
Before proving our main result, we need just two more technical lemmas.
The first one is a simple consequence of the Jordan curve theorem. 
\begin{lem}
\label{lem:winding image}Let $p<2$. Let $\gamma$ and $\tilde{\gamma}$
be two simple curves with finite $p$ variation such that $\gamma_{0}=\tilde{\gamma}_{0}$,
$\gamma_{1}=\tilde{\gamma}_{1}$ and $\eta\left(\tilde{\gamma}\star\overleftarrow{\gamma},\left(x,y\right)\right)=0$
for all $\left(x,y\right)\in\mathbb{R}^{2}\backslash\left(\gamma\left[0,1\right]\cup\tilde{\gamma}\left[0,1\right]\right)$.
Then $\gamma\left[0,1\right]=\tilde{\gamma}\left[0,1\right]$. \end{lem}
\selectlanguage{english}%
\begin{proof}
Assume for contradiction that there exists a $t\in\left(0,1\right)$
such that $\tilde{\gamma}_{\sigma}\notin\gamma\left[0,1\right]$.
Let 
\begin{eqnarray*}
s & := & \inf\left\{ \tau\leq\sigma:\tilde{\gamma}\left[\tau,\sigma\right]\cap\gamma\left[0,1\right]=\emptyset\right\} \\
t & := & \sup\left\{ \tau\geq\sigma:\tilde{\gamma}\left[\sigma,\tau\right]\cap\gamma\left[0,1\right]=\emptyset\right\} .
\end{eqnarray*}

Then $\tilde{\gamma}_{s},\tilde{\gamma}_{t}\in\gamma\left[0,1\right]$
and $s<\sigma<t$. Let $u,v\in\left[0,1\right]$ be such that $\gamma_{u}=\tilde{\gamma}_{s}$
and $\gamma_{v}=\tilde{\gamma}_{t}$. As $\gamma$ and $\tilde{\gamma}$
are both simple, then either $\tilde{\gamma}|_{\left[s,t\right]}\star\gamma|_{\left[u,v\right]}$
or $\tilde{\gamma}|_{\left[s,t\right]}\star\overleftarrow{\gamma|_{\left[u,v\right]}}$
is a simple closed curve. This shows that there exists a simple curve
$\xi$ starting from $\tilde{\gamma}_{s}$ and ending at $\tilde{\gamma}_{t}$
such that $\tilde{\gamma}|_{\left[s,t\right]}\star\overleftarrow{\xi}$
is a simple closed curve. 

By the Jordan curve theorem $\tilde{\gamma}_{\sigma}$ lies in both
the closure of the interior and the closure of the exterior of $\tilde{\gamma}|_{\left[s,t\right]}\star\overleftarrow{\xi}$.
Therefore, for any $\varepsilon>0$, the Euclidean ball centred at
$\tilde{\gamma}_{\sigma}$ with radius $\varepsilon$ contains a point
$x_{\varepsilon}$ in the interior of $\tilde{\gamma}|_{\left[s,t\right]}\star\overleftarrow{\xi}$
and a point $y_{\varepsilon}$ in the exterior of $\tilde{\gamma}|_{\left[s,t\right]}\star\overleftarrow{\xi}$.
Therefore, 
\begin{equation}
\left|\eta\left(\tilde{\gamma}|_{\left[s,t\right]}\star\overleftarrow{\xi},x_{\varepsilon}\right)-\eta\left(\tilde{\gamma}|_{\left[s,t\right]}\star\overleftarrow{\xi},y_{\varepsilon}\right)\right|=1.\label{eq:winding neq}
\end{equation}
 By taking $\varepsilon$ small, $x_{\varepsilon}$ and $y_{\varepsilon}$
will be in the same connected component of 
\[
\mathbb{R}^{2}\backslash\left(\tilde{\gamma}\left[0,s\right]\cup\xi\left[0,1\right]\cup\tilde{\gamma}\left[t,1\right]\cup\gamma\left[0,1\right]\right).
\]
Therefore, 
\begin{eqnarray}
 &  & \eta\left(\tilde{\gamma}|_{\left[0,s\right]}\star\xi\star\tilde{\gamma}|_{\left[t,1\right]}\star\gamma,x_{\varepsilon}\right)\nonumber \\
 & = & \eta\left(\tilde{\gamma}|_{\left[0,s\right]}\star\xi\star\tilde{\gamma}|_{\left[t,1\right]}\star\gamma,y_{\varepsilon}\right).\label{eq:winding eq}
\end{eqnarray}

Combining (\ref{eq:winding neq}) and (\ref{eq:winding eq}) and using
the additivity of winding number we have 
\[
\left|\eta\left(\tilde{\gamma}\star\overleftarrow{\gamma},x_{\varepsilon}\right)-\eta\left(\tilde{\gamma}\star\overleftarrow{\gamma},y_{\varepsilon}\right)\right|=1,
\]
 which is a contradiction.
\end{proof}
The second technical lemma states that the image of a simple curve
determines the curve. 
\selectlanguage{british}%
\begin{lem}
\label{lem:simple image}Let $\gamma$ and $\tilde{\gamma}$ be simple
curves such that $\gamma_{0}=\tilde{\gamma}_{0}$ and $\gamma_{1}=\tilde{\gamma}_{1}$.
If $\gamma\left[0,1\right]=\tilde{\gamma}\left[0,1\right]$, then
there exists a continuous strictly increasing function $r\left(t\right)$
such that 
\[
\gamma_{r\left(t\right)}=\tilde{\gamma}_{t}
\]
for all $t\in\left[0,1\right]$. \end{lem}
\selectlanguage{english}%
\begin{proof}
Let $\gamma^{-1}$ denote the inverse of the function $t\rightarrow\gamma_{t}$,
which exists as $\gamma$ is a simple curve. 

Define a function $r:\left[0,1\right]\rightarrow\left[0,1\right]$
by $r\left(t\right)=\gamma^{-1}\circ\tilde{\gamma}\left(t\right)$.

As both $\gamma$ and $\tilde{\gamma}$ are injective continuous functions
and $\gamma\left[0,1\right]=\tilde{\gamma}\left[0,1\right]$, thus
$r$ is a bijective continuous function from $\left[0,1\right]$ to
$\left[0,1\right]$. Hence it is monotone. 

But $\gamma_{0}=\tilde{\gamma}_{0}$, $\gamma_{1}=\tilde{\gamma}_{1}$,
so $r\left(0\right)=0$ and $r\left(1\right)=1$. Hence $r$ is a
strictly increasing function and the result follows. 
\end{proof}
We now prove Theorem \ref{thm:jordan1}. 

\begin{proof}[Proof of Theorem 3]

The only if direction follows from the invariance of signature under
translation and reparametrisation.

Let $\gamma,\tilde{\gamma}$ be simple curves such that $S\left(\gamma\right)_{0,1}=S\left(\tilde{\gamma}\right)_{0,1}$.
Let $\hat{\gamma}=\tilde{\gamma}+\gamma_{0}-\tilde{\gamma}_{0}$,
and so $\hat{\gamma}_{0}=\gamma_{0}$. By the translation invariance
of signature, $S\left(\hat{\gamma}\right)_{0,1}=S\left(\gamma\right)_{0,1}$.
We want to show that $\hat{\gamma}$ and $\gamma$ are reparametrisations
of each other. 

By Chen's identity, 
\[
S\left(\hat{\gamma}\star\overleftarrow{\gamma}\right)_{0,1}=\mathbf{1}.
\]
Since $\gamma,\hat{\gamma}$ are simple curves, we have by Proposition
\ref{lem:tree-like winding} that 
\[
\eta\left(\hat{\gamma}\star\overleftarrow{\gamma},\left(x,y\right)\right)=0
\]
for all $\left(x,y\right)\in\mathbb{R}^{2}\backslash\hat{\gamma}\star\overleftarrow{\gamma}\left[0,1\right]$. 

Therefore, by Lemma \ref{lem:winding image} and Lemma \ref{lem:simple image},
$\hat{\gamma}$ is a reparametrisation of $\gamma$.\end{proof}

\section{Uniqueness of signature for Schramm-Loewner Evolution}

Let $\left(\Omega,\mathcal{F},\left(\mathcal{F}_{t}\right)_{t\geq0},\mathbb{P}\right)$
be a filtered probability space. Let $\left(B_{t}:t\geq0\right)$
be a one-dimensional standard Brownian motion. Let $0<\kappa$. Let
$z\in\overline{\mathbb{H}}\backslash\left\{ 0\right\} $. For each
$\omega\in\Omega$, consider the initial value problem:

\begin{equation}
\frac{\mathrm{d}g_{t}\left(z,\omega\right)}{\mathrm{d}t}=\frac{2}{g_{t}\left(z,\omega\right)-\sqrt{\kappa}B_{t}\left(\omega\right)}\quad g_{0}\left(z\right)=z\label{2.1}
\end{equation}
We shall recall the following facts about $g_{t}$ from \cite{rhode and schramm}. 
\begin{enumerate}
\item For each $\omega$, a unique solution to this equation exists up to
time $T_{z}>0$, where $T_{z}$ is the first time such that $g_{t}-\sqrt{\kappa}B_{t}\rightarrow0$
as $t\rightarrow T_{z}$. 
\item Define 
\[
H_{t}=\left\{ z\in\mathbb{H}:t<T_{z}\right\} \text{ and }K_{t}=\mathbb{H}\backslash H_{t}
\]
Then $H_{t}$ is open and simply connected.
\item For each time $t>0$, $g_{t}$ defines a conformal map from $H_{t}$
onto $\mathbb{H}$. In particular, $g_{t}$ is invertible.
\item Let $\hat{f}_{t}\left(z\right):=g_{t}^{-1}\left(z+\sqrt{\kappa}B_{t}\right)$.
There exists a $\mathbb{P}$-null set $\mathcal{N}$ such that for
all $\omega\in\mathcal{N}^{c}$, the limit 
\[
\hat{\gamma}\left(t,\omega\right):=\lim_{z\rightarrow0,z\in\mathbb{H}}\hat{f}_{t}\left(z\right)
\]
exists and $t\rightarrow\hat{\gamma}\left(t\right)$ is continuous.
The two dimensional stochastic process $\left(\hat{\gamma}_{t}:t\geq0\right)$
is called the \textit{Chordal SLE$_{\kappa}$ curve}. 
\end{enumerate}
The Loewner correspondence from a continuous path $t\rightarrow B_{t}\left(\omega\right)$
to $t\rightarrow\hat{\gamma}\left(\cdot,\omega\right)$ is in fact
deterministic and one-to-one. Therefore, the measure on the Brownian
paths induces, through this correspondence, a measure on paths in
$\overline{\mathbb{H}}$ from $0$ to $\infty$, which we shall call
the Chordal SLE$_{\kappa}$ measure in $\mathbb{H}$. 

\medskip{}

\begin{thm}
\label{prop:SLE properties}Let $\kappa\leq4$. Let $\mathbb{P}_{\kappa,\mathbb{H}}^{0,\infty}$
be the Chordal SLE$_{\kappa}$ measure in $\mathbb{H}$. Then with
probability one, the following hold: 

1.(\cite{rhode and schramm},Theorem 7.1 and Theorem 6.1)$\hat{\gamma}:[0,\infty)\rightarrow\overline{\mathbb{H}}$
satisfies $\hat{\gamma}_{0}=0$ and $\liminf_{t\rightarrow\infty}\left|\hat{\gamma}_{t}\right|=\infty$. 

2.(\cite{rhode and schramm}, Theorem 6.1)For $0\leq\kappa\leq4$,
$t\rightarrow\hat{\gamma}_{t}$ is a simple curve.
\end{thm}
The fact that $\lim_{t\rightarrow\infty}\hat{\gamma}_{t}=\infty\; a.s.$
means that the signature $S\left(\hat{\gamma}\right)_{0,\infty}$
will not be defined. Therefore, we shall follow \cite{Wer12} and
opt to study the Chordal SLE$_{\kappa}$ curve in the unit disc $\mathbb{D}$,
from $-1$ to $1$. The Chordal $SLE_{\kappa}$ measure in domain
$\mathbb{D}$ with marked points $-1$ and $1$ is defined as follows:
\begin{defn}
\label{SLE}For $\kappa>0$. Let $\mathbb{P}_{\kappa,\mathbb{H}}^{0,\infty}$
be the Chordal SLE$_{\kappa}$ measure in $\mathbb{H}$, $D$ be a
simply connected subdomain of $\mathbb{C}$, $a,b\in\partial D$ and
$f$ be a conformal map from $\mathbb{H}$ to $D$, with $f\left(0\right)=a$
and $f\left(\infty\right)=b$. Then the Chordal SLE$_{\kappa}$ measure
in $D$ with marked points $a$ and $b$ is defined as the measure
$\mathbb{P}_{\kappa,\mathbb{H}}^{0,\infty}\circ f^{-1}$. \end{defn}
\begin{rem}
Although there is a one dimensional family of conformal maps $f$
such that $f$ maps $\mathbb{H}$ to $D$, $0$ to $a$ and $\infty$
to $b$, the scale invariance of the Chordal SLE measure in $\mathbb{H}$
means that the measure $\mathbb{P}_{\kappa,\mathbb{H}}^{0,\infty}\circ f^{-1}$
is the same no matter which member $f$ in this one dimensional family
we use.\end{rem}
\begin{thm}
\label{prop:SLE regularity}(\cite{Wer12}, Section 4.1) Let $0<\kappa\leq4$.
Let $\mathbb{P}_{\kappa,\mathbb{D}}^{-1,1}$ be the Chordal SLE$_{\kappa}$
measure in $\mathbb{D}$ with marked points $-1$ and $1$. Then with
probability one, $\gamma$ has finite $p$ variation for any $p>1+\frac{\kappa}{8}$.
\end{thm}
We now prove our almost sure uniqueness theorem concerning the signature
of SLE curves. 

\begin{proof}[Proof of Theorem 4]Let $D$ be a Dini-smooth bounded
Jordan domain and $a,b\in\partial D$. Let $A$ be the set of curves
$\gamma$ such that 

1. $\gamma\left(0\right)=a$, $\gamma\left(1\right)=b$. 

2. $\gamma$ has finite $\frac{13}{8}$ variation.

3. $\gamma$ is simple. 

Let $\mathbb{P}_{\kappa,D}^{a,b}$ be the Chordal SLE$_{\kappa}$
measure in $D$ with marked points $a$ and $b$. A conformal map
from $\mathbb{D}$ to a bounded Dini-smooth Jordan domain $D$ has
bounded derivative up to the boundary. Therefore, by Theorem \ref{prop:SLE regularity},
the SLE$_{\kappa}$ curves in any bounded Dini-smooth Jordan domain
have finite $p$ variation for any $p>1+\frac{\kappa}{8}$. Moreover,
a conformal map from a Jordan domain $D$ to a Jordan domain $D^{\prime}$
is continuous and injective on $\overline{D}$. Hence by Theorem \ref{prop:SLE properties}
the SLE curve in a Jordan domain $D$ is also a simple curve. Therefore,
as $\frac{13}{8}>1+\frac{4}{8}$, $\mathbb{P}_{\kappa,D}^{a,b}\left(A^{c}\right)=0$
for all $\kappa\leq4$. 

Let $\gamma,\tilde{\gamma}\in A$ be such that $S\left(\gamma\right)_{0,1}=S\left(\tilde{\gamma}\right)_{0,1}$,
then by Theorem \ref{thm:jordan1}, $\gamma$ and $\tilde{\gamma}$
are reparametrisations of each other. \end{proof}

\section{Expected signature and $n$-point functions}

\subsection{$n$-point functions from expected signature}

We will need the following immediate consequence of the shuffle product
formula. 
\begin{lem}
\label{lem:lincom}Let $\left(k_{1},l_{1}\right),...,\left(k_{n},l_{n}\right)\in\mathbb{N}^{2}$.
Then 
\[
\Pi_{i=1}^{n}\mathbf{e}_{1}^{*\otimes k_{i}}\otimes\mathbf{e}_{2}^{*\otimes l_{i}}\left(S\left(\gamma\right)_{0,1}\right)=\mathbf{e}_{1}^{*\otimes k_{1}}\otimes\mathbf{e}_{2}^{*\otimes l_{1}}\sqcup\ldots\sqcup\mathbf{e}_{1}^{*\otimes k_{n}}\otimes\mathbf{e}_{2}^{*\otimes l_{n}}\left(S\left(\gamma\right)_{0,1}\right)
\]
where the operation $\sqcup$ is the shuffle product operation defined
in Proposition \ref{prop:shuffle-1}. \end{lem}
\begin{proof}
This follows from an iterated use of Proposition \ref{prop:shuffle-1}.
\end{proof}
A well-known observable in the theory of SLE is the following sequence
of $n$-point functions:
\begin{defn}
\label{npoint-1}Let $0<\kappa\leq4$. Let $D$ be a bounded Jordan
domain and $a,b\in\partial D$. Let $\mathbb{P}_{\kappa,D}^{a,b}$
denote the Chordal SLE$_{\kappa}$ measure on $D$ with marked points
$a,b$. Let $\Phi\left(\gamma\right)$ denote the concatenation of
$\gamma$ with the positively oriented arc in $\partial D$ from $b$
to $a$. We shall define the $n$-point function associated with the
probability measure $\mathbb{P}_{\kappa,D}^{a,b}$ to be:
\[
\Gamma_{n}\left(x_{1},y_{1},..,x_{n},y_{n}\right)=\mathbb{P}_{\kappa,D}^{a,b}\left[\left(x_{1},y_{1}\right),\ldots,\left(x_{n},y_{n}\right)\in\mbox{Int}\Phi\left(\cdot\right)\right].
\]

\end{defn}
The $n-$point functions for SLE$_{\kappa}$ curves were first studied
by O. Schramm who calculated the $1$-point function explicitly in
terms of hypergeometric functions (see \cite{n point 1}). Although
PDEs can be written down for the $n$-point functions, the analytic
expressions for general $n$ and $\kappa$ are not known. The only
exception is $n=2$, $D=\mathbb{H}$ and $\kappa=\frac{8}{3}$, which
was predicted in \cite{2 point } and computed rigorously in \cite{Bel-Vik11}. 

\begin{proof}[Proof of Theorem 5]Let $A$ be as in the proof of Theorem
\ref{thm:SLEthm}.

Let $\gamma\in A$. Let $\Phi\left(\gamma\right)$ denote the concatenation
of $\gamma$ with the positively oriented arc in $\partial D$ from
$b$ to $a$. As $\Phi\left(\gamma\right)$ is a simple closed curve,
$\eta\left(\Phi\left(\gamma\right),\left(x,y\right)\right)=1_{\mbox{Int}\Phi\left(\gamma\right)}\left(x,y\right)$.
Then by Lemma \ref{thm:main1}, we have for each $\gamma\in A$, for
all $\left(n_{1},k_{1},\ldots,n_{N},k_{N}\right)\in\mathbb{N}^{2N}$
\begin{eqnarray*}
 &  & \Pi_{i=1}^{N}\mathbf{e}_{1}^{*\otimes\left(n_{i}+1\right)}\otimes\mathbf{e}_{2}^{*\otimes\left(k_{i}+1\right)}\left(S\left(\Phi\left(\gamma\right)\right)_{0,1}\right)\\
 & = & C_{\mathbf{n}}\int_{\mathbb{R}^{2N}}\Pi_{i=1}^{N}x_{i}^{n_{i}}y_{i}^{k_{i}}1_{\left(\mbox{Int}\Phi\left(\gamma\right)\right)^{N}}\mathrm{d}x_{1}\mathrm{d}y_{1}\cdots\mathrm{d}x_{N}\mathrm{d}y_{N}
\end{eqnarray*}
where $\left(\mbox{Int}\Phi\left(\gamma\right)\right)^{n}:=\mbox{Int}\Phi\left(\gamma\right)\times\ldots\times\mbox{Int}\Phi\left(\gamma\right)$
($n$ times) and
\[
C_{\mathbf{n},\mathbf{k}}:=\Pi_{i=1}^{N}\frac{\left(-1\right)^{k_{i}}}{n_{i}!k_{i}!}.
\]

By Lemma \ref{lem:lincom}, for all $\left(n_{1},k_{1},\ldots,n_{N},k_{N}\right)\in\mathbb{N}^{2N}$,
\[
\Pi_{i=1}^{N}\mathbf{e}_{1}^{*\otimes n_{i}}\otimes\mathbf{e}_{2}^{*\otimes k_{i}}\left(S\left(\Phi\left(\gamma\right)\right)_{0,1}\right)=\mathbf{e}_{1}^{*\otimes n_{1}}\otimes\mathbf{e}_{2}^{*\otimes k_{1}}\sqcup\ldots\sqcup\mathbf{e}_{1}^{*\otimes n_{N}}\otimes\mathbf{e}_{2}^{*\otimes k_{N}}\left(S\left(\Phi\left(\gamma\right)\right)_{0,1}\right).
\]

By taking linear combinations, we have 
\begin{eqnarray*}
 &  & \int_{\mathbb{R}^{2N}}e^{\sum_{i=1}^{N}\lambda_{i}x_{i}+\mu_{i}y_{i}}\mathbb{E}_{\kappa,D}^{a,b}\left[1_{D^{N}}\right]\mathrm{d}x_{1}\cdots\mathrm{d}y_{N}\\
 & = & \sum_{n_{1},\ldots,n_{N},k_{1}\ldots k_{N}\geq0}\Pi_{i=1}^{N}\left(\lambda_{i}\right)^{n_{i}}\left(-\mu_{i}\right)^{k_{i}}\mathbf{e}_{1}^{*\otimes\left(n_{1}+1\right)}\otimes\mathbf{e}_{2}^{*\otimes\left(k_{1}+1\right)}\sqcup\ldots\\
 &  & \ldots\sqcup\mathbf{e}_{1}^{*\otimes\left(n_{N}+1\right)}\otimes\mathbf{e}_{2}^{*\otimes\left(k_{N}+1\right)}\left(\mathbb{E}\left[S\left(\Phi\left(\gamma\right)\right)_{0,1}\right]\right)
\end{eqnarray*}

The result then follows by noting $\mathbb{E}_{\kappa,D}^{a,b}\left[1_{D^{N}}\left(\cdot\right)\right]=\Gamma_{N}\left(\cdot\right)$.
\end{proof}

As we may determine the signature of $\Phi\left(\gamma\right)$ from
the signature of $\gamma$ using Chen's identity, this formula gives
a relationship between the expected signature of the Chordal SLE measure
and the $n$-point functions.

\subsection{Expected signature from $n$-point functions}

We may ask whether it is possible to obtain the expected signature
from the $n$-point functions. Unfortunately, here we can do no better
than the deterministic case and are only able to obtain an explicit
formula up to the fourth term. To obtain a simpler formula, we choose
to study the Chordal SLE$_{\kappa}$ measure on $\frac{1}{2}\left(1+\mathbb{D}\right)$
so that almost all paths start from $0$. 
\begin{lem}
\label{prop:sig of closed loop}Let $0<\kappa\leq4$. Let $\gamma$
denote the Chordal SLE$_{\kappa}$ curve from $0$ to $1$ in $\frac{1}{2}\left(1+\mathbb{D}\right)$.
Let $\Phi\left(\gamma\right)$ denote the concatenation of $\gamma$
with the upper semi-circle of the unit disc $\frac{1}{2}\left(1+\mathbb{D}\right)$,
oriented in the anti-clockwise direction. Then the level-4 truncated
expected signature of $\Phi\left(\gamma\right)$ is 
\begin{align}
 & \begin{array}{c}
1+\int_{\mathbb{D}}\left(\left[\mathbf{e}_{1},\mathbf{e}_{2}\right]+\left[\mathbf{x}_{1},\left[\mathbf{e}_{1},\mathbf{e}_{2}\right]\right]+\frac{1}{2}\left[\mathbf{x}_{1},\left[\mathbf{x}_{1},\left[\mathbf{e}_{1},\mathbf{e}_{2}\right]\right]\right]\right)\Gamma_{1}\left(\left(x_{1},y_{1}\right)\right)\mathrm{d}x_{1}\mathrm{d}y_{1}\\
+\frac{1}{2}\int_{\mathbb{R}^{4}}\left[\mathbf{e}_{1},\mathbf{e}_{2}\right]\otimes\left[\mathbf{e}_{1},\mathbf{e}_{2}\right]\Gamma_{2}\left(\left(x_{1},y_{1}\right),\left(x_{2},y_{2}\right)\right)\mathrm{d}x_{1}\mathrm{d}y_{1}\mathrm{d}x_{2}\mathrm{d}y_{2}
\end{array}\label{eq:4th levl-1-1-1}
\end{align}
where $\mathbf{x}_{1}=x_{1}\mathbf{e}_{1}+y_{1}\mathbf{e}_{2}$ and
$\mathbf{x}_{2}=x_{2}\mathbf{e}_{1}+y_{2}\mathbf{e}_{2}$, and $\Gamma_{n}$
is the $n$-point function for the Chordal SLE$_{\kappa}$ measure. \end{lem}
\begin{proof}
Let $A$ be the set defined in the proof of Theorem \ref{thm:SLEthm}.

Let $\gamma\in A$. As $\Phi\left(\gamma\right)$ is closed, $\mathbf{e}_{1}^{*}\left(\log S\left(\Phi\left(\gamma\right)_{0,1}\right)\right)=\mathbf{e}_{2}^{*}\left(\log S\left(\Phi\left(\gamma\right)_{0,1}\right)\right)=0$.
Hence by (\ref{eq:4th term}) and a simple computation, $\pi_{4}\left(\log S\left(\Phi\left(\gamma\right)_{0,1}\right)\right)$
can be written as 
\[
\int_{\mathbb{R}^{2}}\left(\left[\mathbf{e}_{1},\mathbf{e}_{2}\right]+\left[x\mathbf{e}_{1}+y\mathbf{e}_{2},\left[\mathbf{e}_{1},\mathbf{e}_{2}\right]\right]+\frac{1}{2}\left[x\mathbf{e}_{1}+y\mathbf{e}_{2},\left[x\mathbf{e}_{1}+y\mathbf{e}_{2},\left[\mathbf{e}_{1},\mathbf{e}_{2}\right]\right]\right]\right)1_{\mbox{Int}\Phi\left(\gamma\right)}\left(x,y\right)\mathrm{d}x\mathrm{d}y
\]

By taking the exponential and writing $x\mathbf{e}_{1}+y\mathbf{e}_{2}$
as $\mathbf{x}$, 
\begin{align}
\pi_{4}\left(S\left(\Phi\left(\gamma\right)_{0,1}\right)\right) & =1+\int_{\mathbb{R}^{2}}\left(\left[\mathbf{e}_{1},\mathbf{e}_{2}\right]+\left[\mathbf{x},\left[\mathbf{e}_{1},\mathbf{e}_{2}\right]\right]+\frac{1}{2}\left[\mathbf{x},\left[\mathbf{x},\left[\mathbf{e}_{1},\mathbf{e}_{2}\right]\right]\right]\right)1_{\mbox{Int}\Phi\left(\gamma\right)}\left(x,y\right)\mathrm{d}x\mathrm{d}y\nonumber \\
 & +\frac{1}{2}\int_{\mathbb{R}^{2}}\left[\mathbf{e}_{1},\mathbf{e}_{2}\right]1_{\mbox{Int}\Phi\left(\gamma\right)}\left(x,y\right)\mathrm{d}x\mathrm{d}y\otimes\int_{\mathbb{R}^{2}}\left[\mathbf{e}_{1},\mathbf{e}_{2}\right]1_{\mbox{Int}\Phi\left(\gamma\right)}\left(x,y\right)\mathrm{d}x\mathrm{d}y.\label{eq:4th term of Jordan curves}
\end{align}

Note that 
\begin{align*}
 & \int_{\mathbb{R}^{2}}\left[\mathbf{e}_{1},\mathbf{e}_{2}\right]1_{\mbox{Int}\Phi\left(\gamma\right)}\left(x,y\right)\mathrm{d}x\mathrm{d}y\otimes\int_{\mathbb{R}^{2}}\left[\mathbf{e}_{1},\mathbf{e}_{2}\right]1_{\mbox{Int}\Phi\left(\gamma\right)}\left(x,y\right)\mathrm{d}x\mathrm{d}y\\
= & \int_{\mathbb{R}^{4}}1_{\mbox{Int}\Phi\left(\gamma\right)\times\mbox{Int}\Phi\left(\gamma\right)}\left(x_{1},y_{1},x_{2},y_{2}\right)\mathrm{d}x_{1}\mathrm{d}y_{1}\mathrm{d}x_{2}\mathrm{d}y_{2}\left[\mathbf{e}_{1},\mathbf{e}_{2}\right]\otimes\left[\mathbf{e}_{1},\mathbf{e}_{2}\right].
\end{align*}

The proof is completed by taking an expectation. 
\end{proof}
Before we calculate the fourth term of the expected signature of SLE
curves, we need the expected signature of a semi-circle with radius
$\frac{1}{2}$. 
\begin{lem}
\label{lem:semi-circle sig}Let $\phi:\left[0,\pi\right]\rightarrow\mathbb{R}^{2}$
be defined by 
\[
\phi\left(t\right):=\begin{cases}
\frac{1}{2}\left(-\cos t,\sin t\right), & t\in\left[0,\pi\right]\\
\frac{1}{2}+\pi-t & t\in\left[\pi,1+\pi\right]
\end{cases}
\]
The first four terms in the signature of $\phi$ are 
\begin{eqnarray}
 &  & \begin{array}{c}
1-\frac{\pi}{8}\left[\mathbf{e}_{1},\mathbf{e}_{2}\right]-\frac{1}{12}\left[\mathbf{e}_{2},\left[\mathbf{e}_{1},\mathbf{e}_{2}\right]\right]-\frac{\pi}{16}\left[\mathbf{e}_{1}.\left[\mathbf{e}_{1},\mathbf{e}_{2}\right]\right]\\
-\frac{5\pi}{256}\left[\mathbf{e}_{1},\left[\mathbf{e}_{1},\left[\mathbf{e}_{1},\mathbf{e}_{2}\right]\right]\right]-\frac{\pi}{256}\left[\left[\left[\mathbf{e}_{1},\mathbf{e}_{2}\right],\mathbf{e}_{2}\right],\mathbf{e}_{2}\right]+\frac{\pi^{2}}{128}\left[\mathbf{e}_{1},\mathbf{e}_{2}\right]\otimes\left[\mathbf{e}_{1},\mathbf{e}_{2}\right]\\
-\frac{1}{24}\left[\mathbf{e}_{1},\left[\mathbf{e}_{2},\left[\mathbf{e}_{1},\mathbf{e}_{2}\right]\right]\right].
\end{array}\label{eq:semi-circle sig}
\end{eqnarray}
\end{lem}
\begin{proof}
By exactly the same computation required to obtain (\ref{eq:4th term of Jordan curves}),we
have, by denoting $\mathbf{x}=\left(x+\frac{1}{2}\right)\mathbf{e}_{1}+y\mathbf{e}_{2}$,
\begin{eqnarray}
\pi_{4}\left(S\left(\phi\star\psi\right)_{0,1}\right) & = & 1-\int_{\mathbb{R}^{2}}\left(\left[\mathbf{e}_{1},\mathbf{e}_{2}\right]+\left[\mathbf{x},\left[\mathbf{e}_{1},\mathbf{e}_{2}\right]\right]+\frac{1}{2}\left[\mathbf{x},\left[\mathbf{x},\left[\mathbf{e}_{1},\mathbf{e}_{2}\right]\right]\right]\right)1_{\mbox{Int}\phi\star\psi}\left(x,y\right)\mathrm{d}x\mathrm{d}y\nonumber \\
 &  & +\frac{1}{2}\int_{\mathbb{R}^{4}}1_{\mbox{Int}\phi\star\psi\times\mbox{Int}\phi\star\psi}\left(x_{1},y_{1},x_{2},y_{2}\right)\mathrm{d}x_{1}\mathrm{d}y_{1}\mathrm{d}x_{2}\mathrm{d}y_{2}\left[\mathbf{e}_{1},\mathbf{e}_{2}\right]\otimes\left[\mathbf{e}_{1},\mathbf{e}_{2}\right]\label{eq:semi-circle 4th term}
\end{eqnarray}
where the negative sign is due to that $\phi\star\psi$ has negative
orientation. 

By changing to polar coordinate and a simple integration we obtain
(\ref{eq:semi-circle sig}). 
\end{proof}

We are now in a position to calculate the first four terms of the
expected signature of SLE$_{\frac{8}{3}}$ curve.

\begin{proof}[Proof of Theorem 6]Let $\Phi$ and $\phi$ be defined
as in Lemmas \ref{prop:sig of closed loop} and \ref{lem:semi-circle sig}.
Then the expected signature of SLE$_{\frac{8}{3}}$ curve in $\frac{1}{2}\left(1+\mathbb{D}\right)$
is 
\begin{equation}
\mathbb{E}_{\frac{8}{3},\frac{1}{2}\left(1+\mathbb{D}\right)}^{0,1}\left(S\left(\cdot\right)_{0,1}\right)=\mathbb{E}_{\frac{8}{3},\frac{1}{2}\left(1+\mathbb{D}\right)}^{0,1}\left(S\left(\Phi\left(\gamma\right)\right)_{0,1}\right)\otimes S\left(\phi\right)_{0,1}\otimes e^{\mathbf{e}_{1}}.\label{eq:big product}
\end{equation}

By the invariance of the distribution of SLE curve under conjugation,
$\mathbf{e}_{i_{1}}^{*}\otimes\mathbf{e}_{i_{2}}^{*}\otimes\mathbf{e}_{i_{3}}^{*}\otimes\mathbf{e}_{i_{4}}^{*}\left(\mathbb{E}_{\frac{8}{3},\frac{1}{2}\left(1+\mathbb{D}\right)}^{0,1}\left(S\left(\gamma\right)_{0,1}\right)\right)=0$
if $\left(i_{1},i_{2},i_{3},i_{4}\right)$ contains an odd number
of $2$s. Therefore, we only need to look at terms with an even number
of $2$s, which can be calculated by substituting (\ref{eq:4th levl-1-1-1})
and (\ref{eq:semi-circle sig}) into (\ref{eq:big product}), to obtain
\begin{eqnarray}
 &  & \left(\int_{\mathbb{R}^{2}}x_{1}y_{1}\Gamma_{1}\left(x_{1},y_{1}\right)\mathrm{d}x_{1}\mathrm{d}y_{1}-\frac{1}{24}\right)\left[\mathbf{e}_{1},\left[\mathbf{e}_{2},\left[\mathbf{e}_{1},\mathbf{e}_{2}\right]\right]\right]\nonumber \\
 &  & +\left(\int_{\mathbb{R}^{2}}y_{1}\Gamma_{1}\left(x_{1},y_{1}\right)\mathrm{d}x_{1}\mathrm{d}y_{1}-\frac{1}{12}\right)\left[\mathbf{e}_{2},\left[\mathbf{e}_{1},\mathbf{e}_{2}\right]\right]\otimes\mathbf{e}_{1}+\frac{\mathbf{e}_{1}^{\otimes4}}{4!}\nonumber \\
 &  & +\frac{1}{2}\int_{\mathbb{R}^{4}}\Gamma_{2}\left(\left(x_{1},y_{1}\right),\left(x_{2},y_{2}\right)\right)\mathrm{d}x_{1}\mathrm{d}y_{1}\mathrm{d}x_{2}\mathrm{d}y_{2}\left[\mathbf{e}_{1},\mathbf{e}_{2}\right]\otimes\left[\mathbf{e}_{1},\mathbf{e}_{2}\right]\nonumber \\
 &  & -\frac{\pi}{8}\int_{\mathbb{R}^{2}}\Gamma_{1}\left(x_{1},y_{1}\right)\mathrm{d}x_{1}\mathrm{d}y_{1}\left[\mathbf{e}_{1},\mathbf{e}_{2}\right]\otimes\left[\mathbf{e}_{1},\mathbf{e}_{2}\right]+\frac{\pi^{2}}{128}\left[\mathbf{e}_{1},\mathbf{e}_{2}\right]\otimes\left[\mathbf{e}_{1},\mathbf{e}_{2}\right].\label{eq:triple tensor}
\end{eqnarray}

The quantity 
\[
\int_{\mathbb{D}}y_{1}\Gamma_{1}\left(x_{1},y_{1}\right)\mathrm{d}x_{1}\mathrm{d}y_{1}
\]
has been calculated in \cite{Wer12} as 
\begin{equation}
\frac{1}{8}\left(\frac{3}{2}-\mathcal{K}\right).\label{eq:first moment}
\end{equation}

We may compute 
\[
\int_{\mathbb{R}^{2}}\Gamma_{1}\left(x_{1},y_{1}\right)\mathrm{d}x_{1}\mathrm{d}y_{1}\;\mbox{and \;}\int_{\mathbb{D}}x_{1}y_{1}\Gamma_{1}\left(x_{1},y_{1}\right)\mathrm{d}x_{1}\mathrm{d}y_{1}
\]
in a similar spirit to \cite{Wer12} to obtain 
\begin{eqnarray}
\int_{\mathbb{R}^{2}}\Gamma_{1}\left(x_{1},y_{1}\right)\mathrm{d}x_{1}\mathrm{d}y_{1} & = & \frac{\pi}{8}\label{eq:zero moment}\\
\int_{\mathbb{R}^{2}}x_{1}y_{1}\Gamma_{1}\left(x_{1},y_{1}\right)\mathrm{d}x_{1}\mathrm{d}y_{1} & = & \frac{3-2\mathcal{K}}{32}.\label{eq:second moment}
\end{eqnarray}

Finally for the term 
\begin{equation}
A:=\int_{\mathbb{R}^{4}}\Gamma_{2}\left(\left(x_{1},y_{1}\right),\left(x_{2},y_{2}\right)\right)\mathrm{d}x_{1}\mathrm{d}y_{1}\mathrm{d}x_{2}\mathrm{d}y_{2}\label{eq:A}
\end{equation}
we note that as $f\left(w\right):=\frac{w}{w+i}$ maps $\mathbb{H}$
to $\frac{1}{2}\left(1+\mathbb{D}\right)$ and $\infty$ to $1$ and
$0$ to $0$, we may obtain by a change of variable 

\begin{eqnarray*}
A & = & \int_{\mathbb{H}\times\mathbb{H}}\Gamma_{2}^{\mathbb{H}}\left(x_{1},y_{1},x_{2},y_{2}\right)\left|f^{\prime}\left(x_{1}+y_{1}i\right)\right|^{2}\left|f^{\prime}\left(x_{2}+y_{2}i\right)\right|^{2}\mathrm{d}x_{1}\mathrm{d}y_{1}\mathrm{d}x_{2}\mathrm{d}y_{2}\\
 & = & \int_{\mathbb{H}\times\mathbb{H}}\frac{\Gamma_{2}^{\mathbb{H}}\left(x_{1},y_{1},x_{2},y_{2}\right)}{\left(x_{1}^{2}+\left(y_{1}+1\right)^{2}\right)^{2}\left(x_{2}^{2}+\left(y_{2}+1\right)^{2}\right)^{2}}\mathrm{d}x_{1}\mathrm{d}y_{1}\mathrm{d}x_{2}\mathrm{d}y_{2}.
\end{eqnarray*}
where $\Gamma_{2}^{\mathbb{H}}$ is the two point function calculated
in \cite{Bel-Vik11}. By substituting the expression for $\Gamma_{2}^{\mathbb{H}}$
and by a change of variable to polar coordinate, we would obtain (\ref{eq:quad integral}).
The result follows by substituting (\ref{eq:zero moment}), (\ref{eq:first moment})
, (\ref{eq:second moment}) and (\ref{eq:A}) into (\ref{eq:triple tensor}).
\end{proof}

\end{document}